\documentclass{amsart}%
\usepackage{amsmath}
\usepackage{amssymb}
\usepackage{amsthm}
\usepackage{amscd}
\usepackage[dvipdfmx]{graphicx}
\usepackage[mathscr]{eucal}
\usepackage{bm}
\makeindex
\newtheorem{Thm}{Theorem}[section]
\theoremstyle{definition}
\newtheorem{Theorem}[Thm]{Theorem}
\newtheorem{Lemma}[Thm]{Lemma}
\newtheorem{Corollary}[Thm]{Corollary}

\newtheorem{Definition}[Thm]{Definition}
\newtheorem{Example}[Thm]{Example}
\theoremstyle{remark}
\newtheorem{Remark}{Remark}
\font\sy=cmsy10

\font\ym=msbm10

\newcommand{\R}{\text{\ym R}}

\newcommand{\C}{\text{\ym C}}

\newcommand{\cF}{{\hbox{\sy F}}}

\newcommand{\cI}{{\hbox{\sy I}}}
\newcommand{\cL}{{\hbox{\sy L}}}

\newcommand{\cR}{{\hbox{\sy R}}}

\newcommand{\sB}{\mathscr B}

\newcommand{\sF}{\mathscr F}

\newcommand{\sH}{\mathscr H}
\newcommand{\sI}{\mathscr I}

\newcommand{\sN}{\mathscr N}

\newcommand{\End}{\hbox{\rm End}}

\renewcommand{\Im}{\,\hbox{Im}\,}
\renewcommand{\Re}{\,\hbox{Re}\,}
\title[Trace Formulas]{Around trace formulas\\ in non-commutative integration}
\author{Shigeru Yamagami}
\begin{document}
\maketitle
\begin{center}
Graduate School of Mathematics
\end{center}
\begin{center}
Nagoya University 
\end{center}
\begin{center} 
Nagoya, 464-8602, JAPAN
\end{center}    

%\tableofcontents 

\begin{abstract}
Trace formulas are investigated in non-commutative integration theory. 
The main result is to evaluate the standard trace of a Takesaki dual 
and, for this, we introduce the notion of interpolator and accompanied boundary objects. 
The formula is then applied to explore a variation of Haagerup's trace formula. 
\end{abstract} 

\section*{Introduction}
The Haagerup's trace formula in non-commutative integration is a key to his whole theory of 
non-commutative $L^p$-spaces (see \cite{H} and \cite{Terp}). Our purpose here 
is to analyse it from the view point of modular algebras (\cite{AAMT}, \cite{MTB}), 
which was originally formulated in terms of 
Haagerup's $L^p$-theory itself. 
So, to circumvent tautological faults and also to fix notations, 
we first describe modular algebras as well as standard Hilbert 
spaces in terms of basic ingredients of Tomita-Takesaki theory. 

The semifiniteness of Takesaki's duals is then established by constructing relevant Hilbert algebras 
as a collaboration of modular algebras and complex analysis. 
Note that the known proofs of the existence of standard traces are not direct; for example, 
it is usually deduced from the innerness of modular automorphism groups combined with 
a reverse Radon-Nikodym theorem such as Pedersen-Takesaki's or Connes'. 

Since our construction of the Hilbert algebras is based on complex analysis, the associated trace can be 
also described in a calculational way. 
To make the setup transparent, we introduce the notion of interpolators 
together with associated boundary operators and vectors. Viewing things this way, the main trace formula 
turns out to be just a straightforward consequence of definitions. 
The Haagerup's trace formula 
is then derived in a somewhat generalized form as a concrete application of our formula. 

The Haagerup's correspondence between normal functionals and relatively invariant measurable operators 
on Takesaki's duals is also established on our streamlines.  

Recall that the standard approach to these problems is 
by the theory of operator-valued weights (\cite{H1979-1}, \cite{H1979-2}) 
coupled with dual weights (\cite{H1978-1}, \cite{H1978-2}), 
which is based on extended positive parts, a notion of metaphysical flavor, and somewhat elaborate. 
Our method may not provide an easy route either but can be applied rather straightforwardly; 
it is just a simple combination of elementary Fourier calculus and complex analytic nature of 
modular stuffs. 

The presentation below originates from the author's old work in 1990, which was addressed 
on the occasion of a satellite meeting of ICM90 held at Niigata University. 
The author would like to express hearty gratitude to Kichisuke Saito for his organization of 
the meeting and these records. 
%8/23,24

\bigskip
\noindent
\textbf{Notation and Convention:} 
The positive part of a W*-algebra $M$ (resp.~its predual $M_*$) is denoted by $M_+$ 
(resp.~$M_*^+$). 

For a positive element $p$ in $M_+$ or $M_*^+$, its support projection in $M$ is denoted by $[p]$. 

For a functional $\varphi \in M_*^+$, 
the associated GNS-vector in the standard Hilbert space $L^2(M)$ of $M$ 
is denoted by $\varphi^{1/2}$ (natural notation though not standard) and the modular operator by 
$\Delta_\varphi$ so that $\Delta_\varphi(a\varphi^{1/2}) = \varphi^{1/2}a$ for $a \in [\varphi]M[\varphi]$. 

For $\varphi,\psi \in M_*^+$, $\sigma^{\varphi,\psi}_t$ stands for the relative modular group 
of $[\varphi]M[\psi]$, which is simply denoted by $\sigma^\varphi_t$ and expresses 
a modular automorphism group of the reduced algebra $[\varphi]M[\varphi]$ when $\varphi = \psi$. 

For convergence in $M$, w*-topology (resp.~s-topology or s*-topology) 
means weak operator topology (resp.~strong operator topology or *strong operator topology) 
as a von Neumann algebra on the standard Hilbert space $L^2(M)$. 

Direct integrals are indicated by $\oint$ instead of ordinary $\int^\oplus$. This is to avoid duplication 
of sum meanings. 

The notion of weights is used in a very restrictive sense: weights are orthogonal sums of 
functionals in $M_*^+$. 

For an interval $I$ contained in $[0,1]$, $T_I$ expresses the tubular domain based 
on an imaginary trapezoid $\{ (x,y) \in \R^2; x \leq 0, y \leq 0, -(x+y) \in I\}$: 
$T_I = \{ (z,w) \in \C^2; \Im z \leq 0, \Im w \leq 0, -(\Im z + \Im w) \in I\}$. 

A function $f: D \to M$ with $D \subset \C$ is said to be w*-analytic (s*-analytic) if 
it is w*-continuous (s*-continuous) and holomorphic when restricted to the interior $D^\circ$. 
Note that topologies are irrelevant for holomorphicity 
because weaker one implies power series expansions in norm. 

For real numbers $\alpha, \beta$, 
\[ 
\alpha \vee \beta = \max\{ \alpha,\beta\}, 
\quad 
\alpha \wedge \beta = \min\{ \alpha,\beta.\}. 
\]

\section{Standard Hilbert Spaces}
Given a faithful $\omega \in M_*^+$, we denote the associated GNS-vector by $\omega^{1/2}$ 
and identify the left and right GNS-spaces by 
the relation $\Delta_\omega^{1/2}(x\omega^{1/2}) = \omega^{1/2}x$, resulting in an $M$-bimodule 
$L^2(M,\omega) = \overline{M \omega^{1/2} M}$ 
with the positive cone $L^2(M,\omega)_+$ and the compatible *-operation given by 
$L^2(M,\omega)_+ = \overline{\{ a\omega^{1/2} a^*; a \in M \}}$ and 
$(a\omega^{1/2} b)^* = b^*\omega^{1/2} a^*$ in such a way that 
these constitute a so-called standard form of $M$. 

The dependence on $\omega$ as well as its faithfulness is then removed 
by the matrix ampliation technique: 
For each $\varphi \in M^+_*$, let $M\otimes \varphi^{1/2}\otimes M$ be a dummy of the algebraic 
tensor product $M\otimes M$, which is an $M$-bimodule in an obvious manner 
with a compatible *-operation defined by the relation 
$(a\otimes \varphi^{1/2}\otimes b)^* = b^*\otimes \varphi^{1/2} \otimes a^*$. 
On the algebraic direct sum 
\[ 
\bigoplus_{\varphi \in M^+_*} M\otimes \varphi^{1/2}\otimes M   
\]
of these *-bimodules, introduce a sesquiliear form by 
\begin{multline*} 
\left( \left. \bigoplus_{j=1}^n x_j\otimes \omega_j^{1/2}\otimes y_j \right| 
\bigoplus_{k=1}^n x_k' \otimes \omega_k^{1/2}\otimes y'_k \right)\\ 
= \sum_{j,k} ([\omega_k] (x'_k)^*x_j \omega_j^{1/2}| \omega_k^{1/2} y_k'y_j^*[\omega_j]), 
\end{multline*}
which is positive because of 
\begin{align*} 
\sum_{j,k} ([\omega_k] x_k^*x_j \omega_j^{1/2}| \omega_k^{1/2} y_ky_j^*[\omega_j])  
&= (X\omega^{1/2}|\omega^{1/2}Y)\\ 
&= (X^{1/2}\omega^{1/2} Y^{1/2}| X^{1/2} \omega^{1/2} Y^{1/2}) \geq 0. 
\end{align*}
Here $\omega = \text{diag}(\omega_1,\dots,\omega_n)$ denotes a diagonal functional 
on the $n$-th matrix ampliation $M_n(M)$ of $M$ and 
\[ 
X = [\omega] 
\begin{pmatrix}
x_1^*\\
\vdots\\
x_n^*
\end{pmatrix}
\begin{pmatrix}
x_1 & \dots & x_n 
\end{pmatrix}
[\omega] 
\quad 
\text{and}
\quad 
Y = [\omega] 
\begin{pmatrix}
y_1\\
\vdots\\
y_n
\end{pmatrix}
\begin{pmatrix}
y_1^* & \dots & y_n^* 
\end{pmatrix}
[\omega]
\]
are positive elments in $[\omega]M_n(M) [\omega]$. 
Recall that $[\omega] = \text{diag}([\omega_1], \dots, [\omega_n])$.  

The associated Hilbert space is denoted by $L^2(M)$ and 
the image of $a \otimes \varphi^{1/2} \otimes b$ in $L^2(M)$ by $a\varphi^{1/2}b$. 
Here the notation is compatible with the one for $L^2(M,\varphi)$ because 
\[ 
[\varphi]M[\varphi] \otimes \varphi^{1/2} \otimes [\varphi]M[\varphi] \ni 
a \otimes \varphi^{1/2} \otimes b \mapsto a\varphi^{1/2}b \in L^2(M,\varphi) 
\]
gives an isometric map by the very definition of inner products. 
Similar remarks are in order for left and right GNS spaces. 

The left and right actions of $M$ are compatible with taking quotients and they are bounded on $L^2(M)$: 
For $a \in M$, 
\[ 
\left\|\bigoplus_j ax_j\otimes \omega_j^{1/2}\otimes y_j \right\|^2  
= (\omega^{1/2}| ZJYJ \omega^{1/2})
\] 
with 
\[ 
0 \leq Z = [\omega] 
\begin{pmatrix}
x_1^*\\
\vdots\\
x_n^*
\end{pmatrix}
a^*a 
\begin{pmatrix}
x_1 & \dots & x_n
\end{pmatrix}[\omega] 
\leq \|a\|^2 X. 
\]
Moreover, these actions give *-representations of $M$: 
$(a\xi|\eta) = (\xi|a^*\eta)$ and $(\xi a|\eta) = (\xi| \eta a^*)$ for 
$\xi, \eta \in L^2(M)$ and $a \in M$, which is immediate from the definition of inner product. 

The *-operation on $L^2(M)$ is also compatible with the inner product: 
\begin{align*} 
\left\| \left(\bigoplus_j x_j\otimes \omega_j^{1/2}\otimes y_j\right)^* \right\|^2  
&= \left\| \bigoplus_j y^*_j\otimes \omega_j^{1/2}\otimes x^*_j \right\|^2 
= (Y\omega^{1/2}| \omega^{1/2}X)\\ 
&= ((\omega^{1/2}X)^*| (Y\omega^{1/2})^*) 
= (X\omega^{1/2}| \omega^{1/2}Y)\\ 
&= \left\| \bigoplus_j x_j\otimes \omega_j^{1/2}\otimes y_j \right\|^2.   
\end{align*}

In this way, we have constructed a *-bimodule $L^2(M)$ of $M$ in such a way that 
$L^2(M,\varphi) \subset L^2(M)$ for each $\varphi \in M^+_*$ and 
the closed subspaces $\overline{M\varphi^{1/2}}$, $\overline{\varphi^{1/2} M}$ 
in $L^2(M)$ are naturally identified with the left and right GNS spaces of $\varphi$ respectively. 
Moreover, for $\varphi, \psi \in M^+_*$, we have 
$[\varphi]\overline{M\psi^{1/2}} = \overline{\varphi^{1/2}M}[\psi]$ in $L^2(M)$, 
which is just a reflection of the fact that the same identification inside 
$L^2(M_n(M),\omega)$ is used in the definition of inner product. 

\section{Modular Algebras}

Recall the definition of (boundary) modular algebra which was introduced in \cite{AAMT}
to `resolve' various cocycle relations in modular theory. 
We shall here describe it without essential use of the notion of weights. 

Let $M$ be a W*-algebra which is assumed to admit a faithful $\omega \in M_*^+$ for the moment. 
The modular algebra $M(i\R)$ of $M$ is then the *-algebra generated by elements in $M$ and symbols 
$\varphi^{it}$ for $\varphi \in M_*^+$ and $t \in \R$ under the conditions that 
(i) $M(i\R)$ contains $M$ as a *-subalgebra, (ii) $\{ \varphi^{it}\}_{t \in \R}$ is a one-parameter group of 
partial isometries satisfying $\varphi^{i0} = [\varphi]$,  
(iii) $\varphi^{it} a = \sigma_t^{\varphi,\psi}(a) \psi^{it}$ for $\varphi,\psi \in M_*^+$, 
$a \in [\varphi]M[\psi]$ and $t \in \R$. 

By utilizing a faithful $\omega \in M_*^+$, it turns out that $M(i\R)$ is *-isomorphic 
to the algebraic crossed product of $M$ by $\{ \sigma^\omega_t\}$ and therefore 
$M(i\R)$ is an algebraic direct sum of $M(it) = M\omega^{it} = \omega^{it}M$, 
where $M(it) = \sum_{\varphi \in M_*^+} M\varphi^{it}M$. 

Thus the modular algebra $M(i\R)$ is 
$i\R$-graded in the sense that $M(it)^* = M(-it)$, $M(is)M(it) = M(i(s+t))$ and $M(i0) = M$.

We say that an element $a \in M$ is \textbf{finitely supported} if $a = [\varphi]a[\varphi]$ for 
some $\varphi \in M_*^+$. Let $M_f$ be the set of finitely supported elements in $M$. 

\begin{Lemma}
$M_f$ is a w*-dense *-subalgebra of 
$M$ and closed under sequential w*-limits in $M$. Moreover, 
\[ 
M_f = \sum_{\varphi \in M_*^+} M[\varphi] = \sum_{\varphi \in M_*^+} [\varphi]M. 
\] 
\end{Lemma}

\begin{proof}
Clearly $M_f$ is closed under the *-operation and $M_f$ is a subalgebra in view of 
$[\varphi]\vee[\psi] \leq [\varphi + \psi]$. The *-subalgebra $M_f$ is then w*-dense in $M$ in view of 
$\vee_{\varphi \in M_*^+} [\varphi] = 1$. 
If $a = a[\varphi]$, $[a\varphi a^*]$ is the left support of $a$ and 
$[\varphi + a\varphi a^*]a[\varphi a\varphi a^*] = a$. 
Let $a$ be a w*-limit of $\{ a_n\}_{n \geq 1}$ in $M_f$ with $[\varphi_n][a_n][\varphi_n] = a_n$ 
for $n \geq 1$. Then, for $\varphi = \sum_{n=1}^\infty 2^{-n}\varphi_n/ \varphi_n(1) \in M_*^+$, 
$[\varphi]a[\varphi]$. 
\end{proof}

Now we relax the existence of faithful functionals in $M_*^+$ and set 
\[ 
M_f(i\R) = \bigcup_{\varphi \in M_*^+} [\varphi]M[\varphi](i\R),  
\]
where the natural inclusions $[\varphi]M[\varphi](i\R) \subset [\psi]M[\psi](i\R)$ for 
$\varphi, \psi \in M_*^+$ satisfying $[\varphi] \subset [\psi]$ are assumed in the union. 

Finally we add formal expressions of the form $\omega^{it} = \sum_{j \in I} \omega_j^{it}$ for families 
$\{ \omega_j \in M_*^+\}_{j \in I}$ of mutually orthogonal supports 
and allow products with elements in $M$
to get $\{ M(it) \}_{t \in \R}$ so that $M_f(it) \subset M(it)$ and $M(0) = M$. 
In what follows, 
a formal sum $\omega = \sum_{j \in I} \omega_j$ is referred to as a \textbf{weight}\index{weight} of $M$. 
A weight $\omega = \sum \omega_j$ 
is said to be \textbf{faithful}\index{faithful weight} if $1 = \sum [\omega_j]$ in $M$. 
Note that any weight is extended to a faithful one and 
$\{ \omega^{it} \}$ is a one-parameter group of unitaries in $M(i\R) = \oplus M(it)$ 
for a faithful $\omega$ and, 
for another choice of a faithful weight $\phi = \sum_{k \in J} \phi_k$ and $a \in M$, 
\[ 
\phi^{it}a\omega^{-it}  = \sum_{j,k} \phi_k^{it} a \omega_j^{-it}
\] 
defines a continuous family of elements in $M$ so that it consists of unitaries when $a=1$ and 
$\sigma^\omega_t(a) = \omega^{it}a\omega^{-it}$ gives an automorphic action of $\R$ on $M$. 

\begin{Remark} 
Here weights are introduced in a formal and restricted way. 
% but more serious treatments
% enable us to find that continuous one-parameter groups of unitaries 
% of the form $\{ u(t) \in M(it)\}$ are one-to-one correspondence 
% with so-called normal semifinite weights on $M$, which is nothing but a paraphrase of 
% Connes' reverse theorem to Radon-Nikodym cocycle-derivatives. 
\end{Remark}

At this stage, we introduce two more classes of modular algebraic stuffs: 
\[ 
M(i\R + 1/2) = \sum_{t \in \R} M(it + 1/2), 
\quad 
M(i\R + 1) = \sum_{t \in \R} M(it + 1), 
\] 
with 
\[ 
M(it + 1/2) = \sum_{\varphi \in M_*^+} M\varphi^{it + 1/2} = \sum_{\varphi \in M_*^+} \varphi^{it + 1/2}M
\] 
and 
\[ 
M(it + 1) = \sum_{\varphi \in M_*^+} M\varphi^{it + 1} = \sum_{\varphi \in M_*^+} \varphi^{it + 1}M
\] 
so that $M(1/2) = L^2(M)$ and $M(1) = M_*$. 

These are $i\R$-graded *-bimodules of $M(i\R)$ in an obvious way 
and we have a natural module map $M(i\R + 1/2)\otimes_{M(i\R)} M(i\R + 1/2) \to M(i\R + 1)$ 
which respects the grading in the sense that $M(is + 1/2)M(it + 1/2) = M(i(s+t) + 1)$ for 
$s,t \in \R$. 
In particular, given a weight $\omega$ on $M$, 
we have $M(it + s) = M(s)\omega^{it} = \omega^{it} M(s)$ for $s=1/2,1$. 

The evaluation of $\varphi \in M_*$ at the unit $1 \in M$ is called the \textbf{expectation} of 
$\varphi$ and denoted by $\langle \varphi \rangle$. Note that the expectation satisfies 
the trace property for various combinations of multiplications such as 
$\langle a\varphi\rangle = \langle \varphi a\rangle$ and 
$\langle \varphi^{it} \xi \psi^{-it}\eta \rangle = \langle \psi^{-it} \eta \varphi^{it}\xi\rangle$ 
for $a \in M$, $\varphi, \psi \in M_*^+$ and $\xi,\eta \in L^2(M)$. 

The scaling $\varphi \mapsto e^{-s}\varphi$ on $M_*^+$ gives rise to a *-automorphic action 
$\theta_s$ of $s \in \R$ (called the \textbf{scaling automorphisms}) on these modular stuffs: 
$\theta_s(x\varphi^{it + r}) = e^{-ist - sr} x\varphi^{it + r}$ for $x \in M$, $r \in \{0, 1/2, 1\}$ 
and $t \in \R$. 

\begin{Remark}
Since elements in $L^2(M)$ and $M_*$ are always `finitely supported', we can decribe $M(it + s)$ 
($s=1/2,1$) without referring to weights. 
\end{Remark}

\section{Analytic Properties} 
We here collect well-known analytic properties of modular stuffs 
(proofs can be found in \cite{BR1} and \cite{Takesaki2} for example). 

\begin{Lemma}[Modular Extension]\label{ME}
For $\varphi, \psi \in M_*^+$ and $a \in M$, $\R \ni t \mapsto \varphi^{it} a \psi^{1-it} \in M_*$ 
is extended analytically to a norm-continuous function $\varphi^{iz}a\psi^{1-iz}$ 
on the strip $-1 \leq \Im z \leq 0$ with a bound 
\[ 
\|\varphi^{it + r} a\psi^{-it + 1-r}\| \leq \| \varphi a\|^r \| a\psi\|^{1-r}
\quad (0 \leq r \leq 1)
\] 
in such a way that 
\[ 
(\varphi^{iz}a\psi^{1-iz})|_{z = t - i} = \varphi^{1+it}a \psi^{-it}, 
\quad 
(\varphi^{iz}a\psi^{1-iz})|_{z = t - i/2} = \varphi^{it + 1/2}a \psi^{-it + 1/2}. 
\] 
\end{Lemma}

\begin{Corollary}[KMS condition] Let $\varphi, \psi \in M^*_+$ and $a \in [\varphi]M[\psi]$. 
Then the function $\sigma^{\varphi,\psi}_t(a)\psi^{1/2} = \varphi^{it} a \psi^{-it} \psi^{1/2}$ 
of $t \in \R$ 
is analytically extended to an $L^2(M)$-valued continuous function $\varphi^{iz}a \psi^{-iz + 1/2}$ of 
$z \in \R - i[0,1/2]$ so that 
$(\varphi^{iz} a \psi^{-iz + 1/2})_{z=t - i/2} = \varphi^{1/2} \varphi^{it}a \psi^{-it} = 
\varphi^{1/2} \sigma^{\varphi,\psi}_t(a)$. 
\end{Corollary} 

\begin{Lemma}\label{majorization1}
Let $\omega \in M_*^+$ be faithful and let $a \in M$. 
Then the following conditions are equivalent. 
\begin{enumerate}
\item 
The inequality $a^*\omega a \leq \omega$ holds in $M_*^+$. 
\item
We can find a function $a(z) \in M$ of $z \in \R - i[0,1/2]$ such that 
$a(t) = \omega^{it} a \omega^{-it}$ for $t \in \R$, $a(z)\xi \in L^2(M)$ is norm-analytic in $z$ 
for any $\xi \in L^2(M)$ and $\| a(-i/2)\| \leq 1$. 
\item 
We can find an element $b \in M$ satisfying $\| b\| \leq 1$ and $\omega^{1/2}a = b \omega^{1/2}$. 
\end{enumerate}
Moreover, if this is the case, with the notation in (ii), $\xi a(z) \in L^2(M)$ is norm-continuous 
in $z$ for every $\xi \in L^2(M)$. 
\end{Lemma}

\begin{Corollary}\label{majorization2}
For $\varphi, \psi \in M_*^+$, the following conditions are equivalent. 
\begin{enumerate}
\item The inequality $\varphi \leq \psi$ holds in $M_*^+$. 
\item $[\varphi] \leq [\psi]$ and the function $\varphi^{it}\psi^{-it}$ of $t \in \R$ is analytically extended 
to an $M$-valued function $\varphi^{iz}\psi^{-iz}$ of $z \in \R - i[0,1/2]$ so that 
$\varphi^{iz}\psi^{-iz}\xi \in L^2(M)$ is norm-continuous in $z$ for any $\xi \in L^2(M)$ and 
$\| \varphi^{1/2}\psi^{-1/2}\| \leq 1$. 
\item We can find an element $c \in M$ satisfying $\| c\| \leq 1$ and $\varphi^{1/2} = c\psi^{1/2}$. 
\end{enumerate}
Moreover, if this is the case, $\xi \varphi^{iz}\psi^{-iz} \in L^2(M)$ is norm-continuous in 
$z \in \R - i[0,1/2]$ for any $\xi \in L^2(M)$. 
\end{Corollary}

\begin{Remark}
Under the above majorization conditions, the relevant analytic extensions 
are norm-bounded as $M$-valued functions of $z \in \R -i[0,1/2]$ thanks to Banach-Steinhaus theorem. 
\end{Remark}

\section{Sectional Continuity}
We now describe continuity properties of families $\{ M(it + r)\}_{t \in \R}$ for $r =0, 1/2$ and $1$. 
Let us begin with a simple observation on the continuity of modular actions: 
Let $\varphi = \sum \varphi_j$ and $\psi = \sum \psi_k$ be weights on $M$ in our sense. 
For $\xi \in L^2(M)$, 
\[ 
\varphi^{it} \xi \psi^{-it} = \sum_{j,k} \varphi_j^{it}\xi \psi_k^{-it}
\]
is norm-continuous in $t \in \R$ as an orthogonal sum of $L^2(M)$-valued norm-continuous functions 
$\varphi_j^{it}\xi \psi_k^{-it}$. As any $\phi \in M_*^+$ has an expression $\xi \eta$ with 
$\xi, \eta \in L^2(M)$, one sees that 
\[ 
\varphi^{it}\phi \psi^{-it} = (\varphi^{it} \xi \omega^{-it}) (\omega^{it} \eta \psi^{-it}) 
\]
($\omega$ being an auxiliary faithful weight) 
is an $M_*$-valued norm-continuous function of $t \in \R$ as a product of $L^2(M)$-valued 
norm-continuous functions. 

The following facts on continuity of sections of $\{ M(it+r)\}$ 
are then more or less straightforward from this observation. 

\begin{Lemma}\label{WC}
For a section $x = \{ x(t)\}$ of $\{ M(it)\}$, the following conditions are equivalent. 
\begin{enumerate}
\item 
There exists a faithful weight $\omega$ on $M$ such that $\omega^{-it}x(t) \in M$ is 
w*-continuous 
in $t \in \R$. 
\item 
There exists a faithful weight $\omega$ on $M$ such that $x(t)\omega^{-it} \in M$ is 
w*-continuous 
in $t \in \R$. 
\item 
For any faithful weight $\omega$ on $M$, $\omega^{-it}x(t) \in M$ is w*-continuous 
in $t \in \R$. 
\item 
For any faithful weight $\omega$ on $M$, $x(t)\omega^{-it} \in M$ is w*-continuous 
in $t \in \R$. 
\item 
For any $\phi \in M_*^+$, $\phi^{-it} x(t) \in M$ is w*-continuous in $t \in \R$. 
\item 
For any $\phi \in M_*^+$, $x(t)\phi^{-it} \in M$ is w*-continuous in $t \in \R$. 
\end{enumerate}
Moreover, if $\{ x(t)\}$ satisfies these equivalent conditions, $\| x(t)\|$ is 
locally bounded in $t \in \R$. 

We say that a section $\{ x(t)\}$ is \textbf{w*-continuous}\index{w*-continuity} 
if it satisfies any of these equivalent conditions.
\end{Lemma}

We here introduce the *-operation on sections by 
\[ 
x^*(t) = x(-t)^* \in M(it+s)
\quad 
\text{for a section $\{x(t) \in M(it+r)\}$.} 
\] 
As a consequence of the above lemma, for a section $x(t) \in M(it)$, 
$x^*(t)$ as well as $ax(t)b$ with $a, b \in M$ are w*-continuous if so is $x(t)$. 

\begin{Lemma}\label{NC}
Let $p=1$ or $2$ with the notation $L^1(M) = M_*$ for $p=1$. 
Then the following conditions on a section $\{ \xi(t)\}$ of $\{ M(it +1/p)\}$ are equivalent. 
\begin{enumerate}
\item 
There exists a faithful weight $\omega$ on $M$ such that $\omega^{-it}\xi(t) \in L^p(M)$ is norm-continuous 
in $t \in \R$. 
\item 
There exists a faithful weight $\omega$ on $M$ such that $\xi(t)\omega^{-it} \in L^p(M)$ is norm-continuous 
in $t \in \R$. 
\item 
For any faithful weight $\omega$ on $M$, $\omega^{-it}\xi(t) \in L^p(M)$ is norm-continuous 
in $t \in \R$. 
\item 
For any faithful weight $\omega$ on $M$, $\xi(t)\omega^{-it} \in L^p(M)$ is norm-continuous 
in $t \in \R$. 
\item 
For any $\phi \in M_*^+$, $\phi^{-it} \xi(t) \in L^p(M)$ is norm-continuous in $t \in \R$. 
\item 
For any $\phi \in M_*^+$, $\xi(t)\phi^{-it} \in L^p(M)$ is norm-continuous in $t \in \R$. 
\end{enumerate}
We say that a section $\{ \xi(t)\}$ is \textbf{norm-continuous}\index{norm-continuity} 
if it satisfies any of these equivalent conditions. 
Notice here that $\xi^*(t) = \xi(-t)^*$ is norm-continuous if so is $\xi(t)$. 
\end{Lemma}

\begin{Definition}
A section $\{ x(t)\}$ of $\{ M(it)\}_{t \in \R}$ is said to be \textbf{s*-continuous} if 
$x(t)\xi$ and $\xi x(t)$ are norm-continuous for any $\xi \in L^2(M)$. 
Notice that $x^*(t)$ is s*-continuous if and only if so is $x(t)$ in view of 
$x^*(t)\xi = (\xi^* x(-t))^*$ and $\xi x^*(t) = (x(-t)\xi^*)^*$. 
Clearly s*-continuous sections are w*-continuous. 

To control the norm of a w*-continuous section $x = \{ x(t) \in M(it) \}$, two norms 
are introduced by 
\[ 
\| x\|_\infty = \sup\{ \| x(t)\|; t \in \R\}, 
\quad 
\| x\|_1 = \int_\R \| x(t)\|\, dt
\]
and $x(t)$ is said to be \textbf{bounded} if $\| x\|_\infty < \infty$ and 
\textbf{integrable}\index{integrable section} if $\| x\|_1 < \infty$. 
Note here that $\| x(t)\|$ is locally bounded and lower-semicontinuous. 
\end{Definition}

\begin{Lemma}\label{2-continuity} 
The following conditions on a section $\{ x(t)\}$ of $\{ M(it)\}$ are equivalent. 
\begin{enumerate}
\item 
For any $\xi \in L^2(M)$, $\{ x(t)\xi \}$ is a norm-continuous section of $\{ M(it+1/2)\}$. 
\item 
For any $\varphi \in M_*^+$ and any $\xi \in L^2(M)$, 
$x(t) \varphi^{-it}\xi \in L^2(M)$ is norm-continuous in $t \in \R$. 
\item 
The norm function $\| x(t)\|$ is locally bounded and, 
for a sufficiently large $\phi \in M_*^+$, 
$x(t) \phi^{-it + 1/2} \in L^2(M)$ is norm-continuous in $t \in \R$, i.e.,  
given any $\varphi \in M_*^+$, we can find $\phi \in M_*^+$ such that 
$\varphi \leq \phi$ 
and $x(t) \phi^{-it + 1/2} \in L^2(M)$ is norm-continuous in $t \in \R$. 
\end{enumerate}
\end{Lemma}

% \begin{proof}
% (i) $\Leftrightarrow$ (ii): Let $\varphi \in M_*^+$. If (i) holds, 
% $\varphi^{-it}x(t) \in M$ is s-continuous in $t \in \R$ and hence 
% \[ 
% x(t)\varphi^{-it}\xi = \Delta^{\varphi}_t (\varphi^{-it} x(t)) \Delta^\varphi_{-t}\xi 
% \] 
% is an $L^2(M)$-valued norm-continuous function of $t \in \R$. 
% Likewise, under the condition (ii), 
% $\varphi^{-it} x(t)\xi = \Delta^\varphi_{-t} (x(t)\varphi^{-it}) \Delta^\varphi_t\xi \in L^2(M)$ is norm-continuous 
% in $t \in \R$. 

% Assume (i) and (ii). The section $x(t)$ is w*-continuous by (i) and 
% the local boundedness of $\| x(t)\|$ follows from Lemma~\ref{WC}, whereas the remaining condition 
% in (iii) is included in (ii). 

% (iii) $\Rightarrow$ (ii): Given $\varphi \in M_*^+$ and $\xi \in L^2(M)$, 
% choose $\phi \in M_*^+$ so that $[\varphi] \leq [\phi]$ and 
% $\phi$ supports $\xi$ on the left. Then 
% \[ 
% \| x(t)\phi^{-it}(\phi^{1/2}a - \phi^{1/2}b) \| \leq \| x(t)\| \| \phi^{1/2}a - \phi^{1/2}b\|, 
% \] 
% together with local boundedness of $\|x(t)\|$, shows that $x(t)\phi^{-it}\eta$ is norm-continuous 
% for any $\eta \in \overline{\phi^{1/2}M}$ 
% as a locally uniform limit of $L^2(M)$-valued norm-continuous functions of $t \in \R$. 
% Now $x(t)\varphi^{-it}\xi = (x(t)\phi^{-it}) (\phi^{it}\varphi^{-it}\xi)$ is norm-continuous in $t \in \R$ 
% as a product of s-continuous operator-valued function $x(t)\phi^{-it}$ on $\overline{\phi^{1/2}M}$ and 
% a norm-continuous function $\phi^{it}\varphi^{-it}\xi$ in $\overline{\phi^{1/2}M}$. 
% \end{proof}

\begin{Corollary}
A section $x(t) \in M(it)$ is s*-continuous if and only if $\| x(t)\|$ is locally bounded 
and $L^2(M)$-valued functions 
$x(t) \phi^{-it+1/2}$, $\phi^{-it+1/2}x(t)$ are norm-continuous for a sufficiently large $\phi \in M_*^+$. 
\end{Corollary} 

A section $\{ x(t) \in M(it)\}_{t \in \R}$ 
is said to be \textbf{finitely supported} if we can find $\phi \in M_*^+$ so that
$x(t) = [\phi]x(t)[\phi]$ for every $t \in \R$. We say that $\{ x(t)\}$ is locally bounded (bounded) if 
so is the function $\| x(t)\|$ of $t$.

\section{Convolution Algebra}
Consider a bounded, s*-continuous and integrable section $\{ f(t) \in M(it)\}$ 
and identify it 
with a formal expression like $\displaystyle \int_\R f(t)\, dt$, which is compatible with the *-operation by 
\[ 
\left( \int_\R f(t)\, dt \right)^* = \int_\R f(t)^*\, dt = \int_\R f(-t)^*\, dt = \int_\R f^*(t)\, dt. 
\] 
Moreover, a formal rewriting 
\[
\int_\R f(s) \, ds \int_\R g(t) \, dt 
= \int_\R \left(\int_\R f(s) g(t-s)\, ds\right)\, dt 
\]
suggests to define a product of $f$ and $g$ by 
\[ 
(fg)(t) = \int_\R f(s) g(t-s)\, ds = \int_\R f(t-s) g(s)\, ds.  
\]
It is then a routine work to check that the totality of such sections constitutes 
a normed *-algebra in such a way that 
\[ 
\| fg\|_\infty \leq (\| f\|_1 \| g\|_\infty) \wedge (\| f\|_\infty \| g\|_1), 
\quad 
\| fg\|_1 \leq \| f\|_1\, \| g\|_1. 
\] 

We notice that 
the scaling automorphism $\theta_s$ on $\{ M(it)\}$ induces a *-automorphic action on the *-algebra 
of sections by $(\theta_sf)(t) = e^{-ist} f(t)$.

Here we shall apply formal arguments to illustrate how tracial functionals can be associated 
to this kind of *-algebras. 

\bigskip
\noindent
\textbf{Formal manipulation is an easy business}: 
Imagine that a section $f(t)$ has an analytic extension to the region 
$-1 \leq \Im z \leq 0$ in some sense so that $f^*(z) = f(-\overline{z})^*$ and define a linear functional by 
\[ 
\tau\left( \int_\R f(t) \, dt \right) = \langle f(-i) \rangle. 
\] 
Note that $f(-i)$ in the right hand side belongs to $M(1) = M_*$. We then have 
\begin{align*} 
\tau(f^*f) &= \int_\R \langle f^*(s) f(-i-s) \rangle\, ds
= \int_\R \langle f^*(s -i/2) f(-s -i/2) \rangle\, ds\\
&= \int_\R \langle f(-s -i/2)^* f(-s -i/2) \rangle\, ds \geq 0, 
\end{align*}
where Cauchy's integral theorem is formally used in the first line. The trace property is seen from 
\begin{align*}
\tau(fg) &= \int \langle f(s) g(-i-s)\rangle \, ds
= \int \langle f(t-i) g(-t) \rangle \,dt\\
&= \int \langle g(-t) f(t-i) \rangle\, dt = \int \langle g(t) f(-t-i) \rangle\, dt 
= \tau(gf), 
\end{align*} 
where Cauchy's integral theorem is again used formally in the first line. 

\bigskip
Going back to the sane track,  
it turns out that it is not easy 
to make all of the above formal arguments rigorous at least 
in a reference-weight-free fashion. 
Instead we shall construct a Hilbert algebra as a halfway business in what follows, 
which is enough to extract the tracial functional.

\section{Hilbert Algebras} 
\begin{Definition}
A section $\{ f(t) \in M(it) \}$ is said to be \textbf{half-analytic} if, for a sufficiently 
large $\phi \in M_*^+$, 
the function $f_\phi(t',t'') = \phi^{-it'} f(t' + t'') \phi^{-it''}$ of $(t',t'') \in \R^2$ 
is analytically extended 
to a bounded $M$-valued s*-continuous 
function $f_\phi(z',z'') = \phi^{-iz'} f(z' + z'') \phi^{-iz''}$ of $(z',z'') \in T_{[0,1/2]}$. 
\end{Definition}

Note here that sufficient largeness in the condition has a meaning: 
For a $\phi$ majorized by $\omega \in M_*^+$, 
$\phi^{it}\omega^{-it}$  is analytically extended to a s*-continuous function
$\phi^{iz}\omega^{-iz}$ of $z \in \R -i[0,1/2]$ (Corollary~\ref{majorization2}) and therefore 
$\omega^{-it'} f(t'+t'') \omega^{-it''}$ has an analytic extension of the form 
$(\omega^{-iz'}\phi^{iz'})(\phi^{-iz'}f(z'+z'')\phi^{-iz''}) (\phi^{iz''} \omega^{-iz''})$,  
which is s*-continuous as a product of s*-continuous locally 
bounded operator-valued functions. 

Note also that the s*-continuity of $\phi^{-iz'}f(z'+z'')\phi^{-iz''}$ 
is equivalent to the norm-continuity of $L^2$-valued functions 
$(\phi^{-iz'}f(z'+z'')\phi^{-iz''})\phi^{1/2}$ and $\phi^{1/2}(\phi^{-iz'}f(z'+z'')\phi^{-iz''})$ 
(Lemma~\ref{2-continuity}). 
These are analytic extensions of 
$\phi^{-it'}f(t'+t'') \phi^{1/2 -it''}$ and $\phi^{1/2-it'}f(t'+t'')\phi^{-it''}$, 
whence simply 
denoted by $\phi^{-iz'}f(z'+z'') \phi^{1/2 -iz''}$ and $\phi^{1/2-iz'}f(z'+z'')\phi^{-iz''}$  respectively.  

\medskip
\noindent
\textbf{Warning}: No separate meaning of $f(z)$ is assigned here. 

\medskip

It is immediate to see that $f(t)$ is half-analytic if and only if so is 
$f^*(t) = f(-t)^*$ in such a way that  
\begin{equation} 
\phi^{-iz'}f^*(z'+z'') \phi^{-iz''} 
= \bigl(\phi^{i\overline{z''}} f(-\overline{z''}-\overline{z'}) \phi^{i\overline{z'}}\bigr)^*. 
\end{equation}

% Let $f(t) \in M(it)$ be a half-analytic section. 
% Then $f^*(t) \in M(it)$ is also half-analytic: 
% \[ 
% \phi^{-it'}f^*(t'+t'')\phi^{-it'' + 1/2}= (\phi^{it'' + 1/2} f(-t''-t')\phi^{it'})^*
% \] 
% and 
% \[ 
% \phi^{-it' + 1/2}f^*(t'+t'')\phi^{-it''}= (\phi^{it''} f(-t''-t')\phi^{it' + 1/2})^*
% \]
% admit  
% \[ 
% \phi^{-iz'}f^*(z'+z'')\phi^{-iz'' + 1/2}= (\phi^{i\overline{z''} + 1/2} 
% f(-\overline{z''}-\overline{z'})\phi^{i\overline{z'}})^*
% \] 
% and 
% \[ 
% \phi^{-iz' + 1/2}f^*(z'+z'')\phi^{-iz''}= (\phi^{i\overline{z''}} 
% f(-\overline{z''}-\overline{z'})\phi^{i\overline{z'} + 1/2})^*
% \] 
% as their analytic extensions respectively, 
% i.e., 
% \begin{equation} 
% \phi^{-iz'}f^*(z'+z'') \phi^{-iz''} 
% = \bigl(\phi^{i\overline{z''}} f(-\overline{z''}-\overline{z'}) \phi^{i\overline{z'}}\bigr)^*. 
% \end{equation}

To get the convolution product in a manageable way, we impose the following decaying condition. 
For a half-analytic section $f(t) \in M(it)$, 
the obvious identity $f_\phi(z'+s',z''+s'') = \phi^{-is'} f_\phi(z',z'') \phi^{-is''}$ shows that 
$\| f_\phi(z',z'')\|$ depends only on $r' = -\Im z'$, $r'' = - \Im z''$ and $t=\Re(z'+z'')$, 
which enables us to introduce 
\[ 
|f|_\phi(t) = \sup\{ \| f_\phi(z',z'')\|; r' \geq 0, r'' \geq 0, r' + r'' \leq 1/2\}.  
\]

A half-analytic section $f(t)$ is said to be of \textbf{Gaussian decay} if, 
for a sufficiently large $\phi \in M_*^+$, we can find $\delta>0$ so that $|f|_\phi(t) = O(e^{-\delta t^2})$.  

Now let $\sN$ be the vector space of half-analytic sections of Gaussian decay, 
which is closed under taking the *-operation by (1). 
It is immediate to see that the scaling automorphisms leave $\sN$ invariant so that 
$\phi^{-iz'}(\theta_sf)(z'+z'')\phi^{-iz''} = e^{-is(z'+z'')} \phi^{-iz'}f(z'+z'')\phi^{-iz''}$. 

Let $f, g \in \sN$. Thanks to the Gaussian decay assumption, 
the convolution product $fg$ has a meaning and $(fg)(t)$ is an s*-continuous section. 
To see $fg \in \sN$, we therefore need to check that it admits 
a half-analytic extension of Gaussian decay. 

% \begin{Lemma} Let $\phi \in M_*^+$, $g(t) \in M(it)$ be a half-analytic section, 
% and choose be a weight $\omega$ which supports $\phi$ and $g(t)$. 
% Then, for $\xi \in L^2(M)$, 
% \[ 
% \sigma^\omega_s(g(z-s) \phi^{-i(z-s)})\xi 
% = \Delta_\omega^{is} \bigl(g(z-s) \phi^{-i(z-s)}\bigr) \Delta_\omega^{-is} \xi \in L^2(M) 
% \] 
% is norm-continuous in $s \in \R$ and $z \in \R -i[0,1/2]$. 
% \end{Lemma} 

% \begin{proof}
% Operator product is continuous on norm-bounded sets with respect to the strong operator topology. 
% \end{proof} 

%We first describe the analytic extension of $\phi^{-it'}(fg)(t'+t'') \phi^{-it''}$.  
Choose an auxiliary weight $\omega$ which supports both $f$ and $g$. Then 
% Let $\xi \in L^2(M)$. 
% \[ 
% \int_\R \phi^{-iz'}f(z'+s) \omega^{-is} 
% \sigma^\omega_s(g(z''-s) \phi^{-i(z''-s)}) \omega^{is} \phi^{-is} \xi\, ds 
% \]
% is well-defined as a Bochner integral and
% \begin{multline*} 
% \left\|\int \phi^{-iz'}f(z' + s) \omega^{-is} \sigma^\omega_s(g(z''-s) 
% \phi^{-i(z''-s)}) \omega^{is} \phi^{-is} \xi\, ds \right\|\\ 
% \leq \| \xi\| \int \| \phi^{-i(z'+s)} f(z'+s)\|\, |g|_\phi(\Re z - s)\, ds 
% \end{multline*}
% shows that 
%A s*-cotinuous $M$-valued analytic function $\phi^{-iz'}(fg)(z'+z'')\phi^{-iz''}$ is well-defined by 
\begin{multline*} 
\phi^{-iz'}((fg)(z'+z'')\phi^{-iz''})\\ 
= \int_\R \phi^{-iz'}f(z'+s) \omega^{-is} \sigma^\omega_s(g(z''-s) \phi^{-i(z''-s)}) 
\omega^{is} \phi^{-is} \, ds 
\end{multline*}
gives the s*-continuous analytic extension with its norm estimated by
\begin{multline*}
\| \phi^{-iz'}(fg)(z'+z'') \phi^{-iz''}\|\\ 
\leq \int_\R |f|_\phi(\Re z' +s)\, |g|_\phi(\Re z'' - s)\, ds
= O(e^{-\epsilon\delta t^2/(\epsilon + \delta)}) 
\end{multline*}
for $t = \Re(z'+z'')$ if $|f|_\phi(t) = O(e^{-\epsilon t^2})$ and $|g|_\phi(t) = O(e^{-\delta t^2})$. 

% To see the continuity of $\phi^{-iz'}(fg)(z'+z'') \phi^{-iz''}$ on 
% $(z',z'') \in T_{[0,1/2]}$, consider  
% \begin{multline*} 
% \phi^{-iz'}(fg)(z'+z'') \phi^{-iz'' + 1/2} =\\ 
% \int \phi^{-iz'}f(z'+s) \omega^{-is} \omega^{is}
% \bigl(g(z''-s) \phi^{-i(z''-s)}\phi^{1/2}\bigr) \phi^{-is}\, ds.  
% \end{multline*}
% In view of the norm estimate 
% \begin{multline*}
% \| \phi^{-iz'}f(z'+s) \omega^{-is} \omega^{is}\bigl(g(z''-s) \phi^{-i(z''-s)}\phi^{1/2}\bigr) 
% \phi^{-is}\|\\  
% \leq \phi(1)^{1/2} \| |g|_\phi\|_\infty\, |f|_\phi(\Re(z')+s)
% \end{multline*}
% of the integrand 
% and the integrability of $|f|_\phi(\Re(z')+s)$, 
% the above Bochner integral is norm-continuous in $(z',z'')$ 
% as a uniform limit of $L^2(M)$-valued norm-continuous functions of $(z',z'') \in T_{[0,1/2]}$. 

% Likewise $\phi^{-iz' + 1/2} (fg)(z'+z'') \phi^{-iz''}$ is continuous in $(z',z'') \in T_{[0,1/2]}$. 

% From the inequality 
% \begin{multline*} 
% \| \phi^{-iz'} (fg)(z'+z'') \phi^{-iz'' + 1/2}a - \phi^{-iz'}(fg)(z'+z'') \phi^{-iz'' + 1/2}b\|\\ 
% \leq 
% \| |g|_\phi\|_\infty \| \phi^{1/2} a - \phi^{1/2}b\| \int_\R |f|_\phi(s)\, ds,  
% \end{multline*}
% one sees that, when $\|\phi^{1/2}a - \xi\| \to 0$,  
% $\phi^{-iz'}(fg)(z'+z'') \phi^{-iz''} \phi^{1/2}a$ norm-converges to 
% $\phi^{-iz'}(fg)(z'+z'') \phi^{-iz''}\xi$ uniformly in $(z',z'')$ 
% and the norm-continuity of $\phi^{-iz'}(fg)(z'+z'') \phi^{-iz''}\xi$ follows. 

So far $\sN$ is shown to be a *-algebra with an automorphic action of $\R$ by scaling automorphisms.
%so that the scaling automorphisms gives a *-automorphic action of $\R$ on $\sN$. 
We next introduce an inner product which makes $\sN$ into a Hilbert algebra. 

\begin{Lemma}
The following identity holds for $f \in \sN$ and sufficiently large $\phi, \varphi \in M_*^+$. 
\[ 
[\varphi]\bigl(f(t-i/2)\phi^{-it - 1/2}\bigr) \phi^{it + 1/2} 
= \varphi^{it + 1/2} \bigl(\varphi^{-it - 1/2} f(t - i/2)\bigr) [\phi]
\] 
(the left hand side is therefore depends only on $[\varphi]$ while the right hand side 
depends only on $[\phi]$ and the common element in $M(it + 1/2)$ is reasonably denoted by 
$[\varphi]f(t-i/2)[\phi]$). 
\end{Lemma}

\begin{proof}
For $a \in M$, the identity
\[
\langle (f(t) \phi^{-it}) \phi^{1/2} \sigma^{\phi,\varphi}_t(a) \varphi^{1/2} \rangle 
= \langle \varphi^{1/2}(\varphi^{-it} f(t)) \phi^{1/2} a \rangle 
\] 
is analytically continued from $t$ to $t - i/2$ to get 
\[ 
\langle (f(t-i/2) \phi^{-it-1/2}) \phi^{1/2} \phi^{1/2} \sigma^{\phi,\varphi}_t(a) \rangle 
= \langle \varphi^{1/2} (\varphi^{-it-1/2} f(t-i/2)) \phi^{1/2} a \rangle  
\] 
(use the KMS-condition at $\sigma^{\phi,\varphi}_t(a)\varphi^{1/2}$) and, after a simple rewriting, 
\[ 
\langle (f(t-i/2) \phi^{-it-1/2}) \phi^{it+1/2} \phi^{1/2}a\varphi^{-it} \rangle 
= \langle \varphi^{it +1/2} (\varphi^{-it-1/2} f(t-i/2)) \phi^{1/2} a \varphi^{-it} \rangle.   
\] 
\end{proof}

Since $[\varphi] f(t-i/2) [\phi] = [\varphi] ([\varphi'] f(t-i/2) [\phi']) [\phi]$
whenever $[\varphi] \leq [\varphi']$ and $[\phi] \leq [\phi']$, 
$\| [\varphi] f(t-i/2) [\phi]$ is increasing in $[\varphi]$ and $[\phi]$. 
We claim that 
\[ 
f(t-i/2) = \lim_{\substack{[\varphi] \to 1\\ [\phi] \to 1}} [\varphi] f(t-i/2) [\phi]
\] 
exists in $M(it+1/2)$. 
In fact, if not, we can find increasing sequences $\varphi_n$ and $\phi_n$ in $M_*^+$ 
so that $\displaystyle \lim_{n \to \infty} \| [\varphi_n] f(t-i/2) [\phi_n]\| = \infty$, which 
contradicts with 
\[ 
\| [\varphi_n] f(t-i/2) [\phi_n]\| \leq \| [\varphi] f(t-i/2) [\phi]\| < \infty
\] 
for the choice $\varphi= \sum \varphi_n/2^n\| \varphi_n\|$, 
$\phi= \sum \phi_n/2^n\| \phi_n\|$. 

Moreover, the same reasoning reveals that we can find $\varphi, \phi \in M_*^+$ so that 
$f(t-i/2) = [\varphi] f(t-i/2) = f(t-i/2)[\phi]$. 
Consequently, $\{ f(t-i/2) \in M(it + 1/2) \}$ is a norm-continuous section of Gaussian decay from 
the expression 
\[ 
f(t-i/2) = f(t-i/2)[\phi] = \bigl(f(t-i/2)\phi^{-it-1/2}\bigr) \phi^{it+1/2} 
\]
which is valid for a sufficiently large $\phi$.

\begin{Remark}
By an analytic continuation, 
one sees that any half-analytic section $\{ f(t)\}$ of $\{ M(it)\}$ 
is finitely supported in the sense that there exists $\phi \in M_*^+$ satisfying 
$f(t) = [\phi] f(t) [\phi]$ for every $t \in \R$. 
\end{Remark} 

\begin{Example}
Let $\varphi \in M_*^+$ and $a, b \in [\varphi]M[\varphi]$ be entirely analytic for $\sigma^\varphi_t$. 
Then, for $\alpha>0$ and $\beta \in \C$, 
$f(t) = e^{-\alpha t^2 + \beta t} a \varphi^{it}b$ belongs to $\sN$ and its boundary section is 
$f(t-i/2) = e^{-\alpha(t-i/2)^2 + \beta(t-i/2)} a \varphi^{it + 1/2} b$. 
\end{Example} 

The inner product is now introduced by 
\[ 
(f|g) = \int_\R (f(t-i/2) | g(t-i/2))\, dt 
= \int_\R \langle f(t-i/2)^* g(t-i/2) \rangle\, dt,  
\] 
which is clearly positive-definite and 
the completed Hilbert space $\sH$ is naturally identified with the direct integral 
\[ 
\sH = \oint_\R M(it + 1/2)\, dt 
\] 
because $\sN$ provides a dense set of measurable sections in the right hand side. 
The Hilbert space $\sH$ is then made into a *-bimodule of $M(i\R)$ by 
\[ 
a\omega^{is} \oint_\R \xi(t)\, dt = 
\oint_\R a\omega^{is} \xi(t-s)\, dt   
\] 
and 
\[ 
\left( \oint \xi(t)\, dt \right)^*  = \oint_\R \xi(-t)^*\, dt 
\] 
in such a way that actions of $M(it)$ on $\sH$ are s*-continuous. 

Since the family $\{ M(it+1/2)\}$ is trivialized by obvious isomorphisms 
$L^2(M) \omega^{it} \cong L^2(M) \cong \omega^{it}L^2(M)$ in terms of a faithful weight $\omega$ on $M$, 
we have identifications 
$\sH \cong L^2(M)\otimes L^2(\R)$ in two ways, which transforms left and right multiplications of 
$\omega^{it}$ into a translational unitary by $t \in \R$. 
Recall that our weights are orthogonal direct sums of bounded functionals and 
the multiplication of $\omega^{is}$ on $\sH$ gives 
a continuous one-parameter group of unitaries. 

With these observations in mind, it is immediate to check the axioms of Hilbert algebra: 
the left and right multiplications 
are bounded with respect to the inner product, $\sN^2$ is dense in $\sH$ and 
$(f^*|g^*) = (g|f)$ for $f, g \in \sN$. 

\begin{Remark}
Note that the scaling automorphism $\theta_s$ satisfies 
$(\theta_sf)(t-i/2) = e^{-ist - s/2} f(t-i/2)$ and hence scales the inner product: 
$(\theta_sf|\theta_sg) = e^{-s} (f|g)$ for $f,g \in \sN$. 
\end{Remark}

% Let $f \in \sN$ be supported by $\phi \in M_*^+$. Then, for $h \in \sN$, $hf$ is right-supported by 
% $\phi$ and we have 
% \begin{align*} 
% (hf)(t-i/2) &= \bigl( (hf)(t-i/2) \phi^{-it-1/2} \bigr) \phi^{it + 1/2}\\ 
% &= \int ds\, (h(s)\omega^{-is}) \sigma^\omega_s\bigl(f(t-s-i/2) \phi^{-i(t-s-i/2)}\bigr) 
% \omega^{is}\phi^{-is} \phi^{it+1/2}\\
% &= \int ds\, h(s) \bigl( f(t-s-i/2) \phi^{-i(t-s) - 1/2} \bigr) \phi^{i(t-s) + 1/2}\\
% &= \int ds\, h(s) f(t-s-i/2).
% \end{align*}
% In other words, 
% \[ 
% \oint (hf)(t-i/2)\, dt = \int h(s)\, ds \oint f(t-i/2)\, dt
% \] 
% with 
% \[ 
% \| hf\|_\sH \leq \| f\|_\sH \int \| h(s)\|\, ds. 
% \] 
% Thus the left multiplication of $h$ is bounded with respect to the inner product in such a way 
% that $(hf|g) = (f|h^*g)$ for $f,g \in \sN$. 

% Moreover, with the choice $h(t) = \sqrt{\alpha/\pi} e^{-\alpha t^2} \phi^{it}$, 
% $\lim_{\alpha \to \infty}\|hf - f\|_\sH \to 0$ shows the density of $\sN \sN$ in $\sH$. 

% Finally $(f^*|f^*) = (f|f)$ follows from $f^*(t-i/2) = \bigl(f(-t-i/2)\bigr)^*$. 

% As a conclusion, 
In this way,  we have constructed a Hilbert algebra $\sN$. 
The associated von Neumann algebra is denoted by $N = M\rtimes\R$ and referred to as 
the \textbf{Takesaki dual}\index{Takesaki dual} of $M$ in what follows. 
The scaling automorphisms $\theta_s$ 
of $\sN$ induce a *-automorphic action (also denoted by $\theta_s$)  
of $\R$ on $N$ by $\theta_s(l(f)) = l(\theta_s f)$, 
which is referred to as the \textbf{dual action}. 
Here $l(f)$ denotes a bounded operator on $\sH$ defined by $l(f)g = fg$ for $g \in \sN$. 

Let $\omega$ be a faithful weight on $M$. From the convolution form realization of $\sN$ on 
$\sH$, one sees that $N$ contains $M$ as well as $\omega^{it}$ as operators by left multiplication and 
these in turn generates $N$. Likewise right multiplications of $M$ and $\omega^{it}$ generates 
the right action of $N$ on $\sH$. Thus the Takesaki dual of $M$ is isomorphic to the crossed product 
of $M$ with respect to the modular automorphism group $\{ \sigma^\omega_t\}$, which justifies our notation 
$M\rtimes \R$ for $N$. 

We record here the following well-known fact for later use together with a proof 
to illustrate how the essence can be easily captured in the modular algebra formalism. 

\begin{Theorem}[Takesaki] 
The fixed-point algebra $N^\theta$ of $N$ under the dual action $\theta$ 
is identified with $M$. 
\end{Theorem}

\begin{proof}
Through $\sH \cong L^2(M)\otimes L^2(\R)$ adapted to 
the trivialization $M(it + 1/2)\omega^{-it} = L^2(M)$ of $M(it + 1/2)$, 
the right action of $\omega^{is}$ is realized on $L^2(\R)$ by translations 
whereas $\theta_s$ by multiplication of $e^{-ist}$ on $L^2(\R)$. 
Since these generate $\sB(L^2(\R))$ (Stone-von Neumann), $N^\theta$ is identified with 
$(\sB(L^2(M))\otimes 1) \cap \End(\sH_M)$. 
Let $a \in M$ and $f,g \in L^2(\R)$. 
For $\xi, \eta \in L^2(M)$, 
\[ 
(\eta\otimes g|(\xi\otimes f)a) = (\eta|\xi \sigma^\omega_{\overline g f})
\quad 
\text{with}
\quad 
\sigma^\omega_{\overline g f} = \int_\R \overline{g(t)} f(t)  \omega^{it}a\omega^{-it} \in M
\] 
shows that $T \in \sB(L^2(M))$ belongs to $N^\theta$ if and only if it is in the commutant 
of the right action of $\{ \sigma^\omega_h(a); h \in L^1(\R)\}$ on $L^2(M)$. 
Since $\{ \sigma^\omega_h(a); a \in M, h \in L^1(\R)\}$ generates $M$, this implies $N^\theta \subset M$. 
\end{proof}

Now we introduce some notations and conventions in connection with our Hilbert algebra: 
$\sN$ is regarded as a *-subalgebra of $N$ and we write 
$\sN \tau^{1/2} = \tau^{1/2} \sN$ to indicate 
the corresponding subspace in $\displaystyle \sH = \oint_\R M(it+1/2)\, dt$,  
where $\tau^{1/2}$ is just a dummy symbol but its square $\tau$ will be soon identified with the standard 
trace on $N$. 
Thus $h \in \sN$ is identified with an operator on $\sH$ satisfying $h(f\tau^{1/2}) = (hf)\tau^{1/2}$ 
for $f \in \sN$, whereas  $\displaystyle f \tau^{1/2} = \tau^{1/2} f = \oint_\R f(t-i/2)\, dt$. 

Let $B \supset \sN$ be a dense *-ideal of $N$ such that $B \tau^{1/2} = \tau^{1/2} B$ 
is the set of bounded vectors in $\sH$; $y \in N$ belongs to $B$ if and only if there exists a vector 
$\eta \in \sH$ satisfying $\eta f = y (f\tau^{1/2}) = y(\tau^{1/2}f)$ for any $f \in \sN$ and, 
if this is the case, we write $\eta = y\tau^{1/2} = \tau^{1/2} y$. 
Recall that the standard trace $\tau$ on $N_+$ is defined by 
$\tau(y^*y) = (y\tau^{1/2}|y\tau^{1/2})$ if $y \in B$ and $\tau(y^*y) = \infty$ otherwise. 
Note that, for $f,g \in \sN$, $f^*g \in \sN^2$ is in the trace class and its trace is calculated by 
\[ 
\tau(f^*g) = \int_\R (f(t-i/2)| g(t-i/2))\, dt = (f\tau^{1/2}|g\tau^{1/2}),  
\]
which justifies our notation $f\tau^{1/2}$. 

From the scaling relation $(\theta_s f)(t-i/2) = e^{-ist - s/2} f(t-i/2)$, the inner product is scaled 
by a factor $e^{-s}$ under the *-automorphism of $\sN$ and hence the associated trace $\tau$ scales 
like $\tau(\theta_s(y^*y)) = e^{-s} \tau(y^*y)$ for $y \in N$. 

To each $\xi, \eta \in \sH$, a sesquilinear element $\xi^*\eta \in N_*$ is associated by 
$\langle \xi^*\eta,x\rangle = (\xi|\eta x)$ and $a^*b \tau = \tau a^*b \in N_*$ is defined to be 
$(a\tau^{1/2})^*(b\tau^{1/2})$ for $a,b \in B$. 

As a square root of this correspondence, we have a unitary map $\sH \to L^2(N)$ in such a way that 
$|a|\tau^{1/2} \mapsto (a^*a\tau)^{1/2}$ for $a \in B$. 
Therefore, if we set $B_+ = B \cap N_+$, the closure of $B_+\tau^{1/2} = \tau^{1/2} B_+$ in 
$\sH$ corresponds to the positive cone $L^2(N)_+$. 

Related to these, we recall the following well-known and easily proved fact (cf.~\cite{S} Corollary~19.1). 

\begin{Lemma}\label{HS}
The Hilbert space $\sH$ is canonically isomorphic to 
the vector space of Hilbert-Schmidt class operators with respect to $\tau$ in such a way that 
$\tau(y^*y) = (y\tau^{1/2}| y\tau^{1/2})$. 
Note that a closed operator $y$ affiliated to $N$ is in the Hilbert-Schmidt class if and only if 
$\tau(y^*y) < \infty$. 
\end{Lemma}

% \begin{proof}
% For $\eta \in \sH$, define a linear operator $y: \sN \tau^{1/2} \to \sH$ by $y(g\tau^{1/2}) = \eta g$.  
% Then $y^*(h\tau^{1/2}) = (h^*\eta)^*$ and $l(\eta) = y^{**}$ is a closed operator on $\sH$ affiliated to $N$ 
% such that $l(\eta)^* = l(\eta^*)$. 
% Let $l(\eta) l(\eta)^* = \int_\R \lambda\, E(d\lambda)$ be a spectral decomposition 
% and set $e_n = E([0,n]) \in N$. 
% Then $\eta_n = e_n\eta \in B\tau^{1/2}$ with $l(\eta_n) = e_ny \in N$ shows that 
% \[ 
% \tau(l(\eta)^*l(\eta)) = \lim \tau(l(\eta)^*e_nl(\eta)) = \lim \tau(l(\eta_n)^*l(\eta_n)) 
% = \lim (\eta_n|\eta_n) = (\eta|\eta) < \infty.
% \]
% Consequently $l(\eta)$ is $\tau$-measurable and by polarization $l(\xi)^*l(\eta)$ is in the $\tau$-trace 
% class with $\tau(l(\xi)^*l(\eta)) = (\xi|\eta)$ for $\xi, \eta \in \sH$. 
% \end{proof}

\section{Trace Formula}
We shall now utilize the Hilbert algebra structure behind $N$ to set up a method modeled after $\sN$ 
to calculate the standard trace $\tau$ on $N$. 

Given an open interval $I \subset [0,1/2]$, 
let $\widetilde{\cF}_I$ be the set of $M$-valued analytic functions of $z \in \R -iI$ and set 
$\cF_I = \cup_{\phi \in M_*^+} [\phi]\widetilde{\cF}_I[\phi]$. We write 
$f_\phi(z) \phi^{iz}$ for $\phi \in M_*^+$ and $f_\phi \in [\phi]\sF_I[\phi]$ 
to indicate dummies of elements in $\cF_I$. 
All such dummies are then identified 
by the relation $\varphi^{iz} = (\varphi^{iz}\psi^{-iz})\psi^{iz}$ whenever $\varphi \leq \psi$ and 
the obtained quotient set (which is a kind of inductive limit of dummy elements) 
is denoted by $\cL\cI_I$ and an element in $\cL\cI_I$ is called a left interpolator on $I$. 

Thus each left interpolator is of the form 
$f(z) = f_\varphi(z) \varphi^{iz}$ and we say that $f(z)$ is supported by $\varphi$. 
Then, for $\phi \in M_*^+$ mojorizing $\varphi$, $f(z)$ is supported by $\phi$ and 
$f_\phi(z) = f_\varphi(z) (\varphi^{iz} \phi^{-iz})$, 
which is also denoted by $f(z)\phi^{-iz}$. 

Clearly we have a similar notion of right interpolators with the obvious notations for them. 
These are related by the *-operation defined by $f^*(z) = f(-\overline{z})^*$: 
If $f \in\cL\cI_I$, $f^* \in \cR\cI_I$ so that $\phi^{-iz} f^*(z) = (f(-\overline{z}) \phi^{i\overline{z}})^*$. 

A pair $(l(z),r(z))$ of left and right interpolators on $I$ is called an \textbf{interpolator} 
if one can find $\phi \in M_*^+$ which supports $l$, $r$ and interrelates them in the following sense: 
For each $w \in \R - iI$, the function $\sigma^\phi_t(\phi^{-iw} r(w))$ of $t \in \R$ 
is analytically extended up to the horizontal line $w + \R$ so that the function
$\sigma^\phi_z(\phi^{-iw} r(w))$ is w*-analytic on $D = \{ (z,w) \in \C^2; w \in \R -iI, 
\Im w \leq \Im z \leq 0\}$ and satisfies 
$\sigma^\phi_w(\phi^{-iw} r(w)) = l(w) \phi^{-iw}$. 
Here, for $z \in \C \setminus \R$ and $a \in M$, 
$\sigma_z(a)$ means that $\sigma_t(a)$ ($t \in \R$) is analytically extended 
to a w*-continuous function of $\zeta \in \R + i\Im z[0,1]$ and it is evaluated at $\zeta = z$. 

Since analytical extensions are moved back to 
the starting horizontal lines, the condition is symmetrical in the left-and-right: 
$\sigma^\phi_{-t}(l(w)\phi^{-iw})$ is analytically extended to $\sigma^\phi_{w-z}(\phi^{-iw} r(w))$, 
which is w*-continuous in $(z,w) \in D$. For $(z,w) \in T_I$, the relation 
$\sigma^\phi_{z+w}(\phi^{-i(z+w)}r(z+w)) = l(z+w) \phi^{-i(z+w)}$ is then rewritten into 
$\sigma^\phi_w (\phi^{-i(z+w)}r(z+w)) = \sigma^\phi_{-z}(l(z+w) \phi^{-i(z+w)})$, 
which is a w*-analytic function of $(z,w) \in T_I$ and denoted by 
$\phi^{-iz} f(z+w) \phi^{-iw}$ when $(l(z),r(z))$ is symbolically expressed by $f(z)$. 

Moreover, the interrelating condition is compatible with the majorization changes: 
Let $\phi \leq \omega$ and $z \in \R - iI$. Then 
\[ 
\sigma^\omega_t(\omega^{-iz} r(z)) = (\omega^{-i(z-t)} \phi^{i(z-t)}) 
\sigma^\phi_t(\phi^{-iz} r(z)) \phi^{it} \omega^{-it}
\] 
is analytically continued from $t$ to $z$ to get 
$(l(z)\phi^{-iz})(\phi^{iz}\omega^{-iz}) = l(z) \omega^{-iz}$. 

We say that an interpolator $f(z) = (l(z),r(z))$ is supported by $\phi \in M_*^+$ if 
both $l(z)$ and $r(z)$ are supported by $\phi$ and, in that case, we write 
%$\phi^{-iz}r(z)$ and $l(z)\phi^{-iz}$ have meanings with notations 
${}_\phi f(z) = \phi^{-iz}f(z) = \phi^{-iz} r(z)$ and 
$f_\phi(z) = f(z) \phi^{-iz} = l(z)\phi^{-iz}$. 

Let $\cI_I$ be the set of interpolators on $I$. %When $I = (0,1/2)$, we simply write $\cI$. 
By restriction or extension, $\cI_J \subset \cI_I$ if $I \subset J \subset (0,1/2)$. 
The *-operation on $\cI_I$ is defined by $(l(z),r(z))^* = (r^*(z),l^*(z))$ so that it is compatible with 
the inclusions $\cI_J \subset \cI_I$. 
Notice that $\sN$ can be regarded as a *-subspace of $\cI_{(0,1/2)}$. 

Given an asymptotic function $\rho: \R \setminus [-R,R] \to [0,\infty)$ 
with $R>0$ a positive real, 
an interpolator $f(z)$ on $I$ is said to have a $\rho$-growth and denoted by $f(z) = O(\rho(\Re z))$ 
if we can find $C>0$ so that $\| \phi^{-iz} f(z+w) \phi^{-iw}\| \leq C \rho(\Re(z+w))$ for any $(z,w) \in T_I$ 
satisfying $z+w \in \R \setminus [-R,R] - iI$. 
Note that the growth condition is well-defined thanks to the half-power analyticity for majorization. 

An interpolator $f$ is said to be 
of sub-gaussian growth %(polynomial grouw) 
if, for any small $\epsilon>0$, 
$f(z)\phi^{-iz} = O(e^{\epsilon (\Re z)^2})$. %($=O(p(\Re z))$ with $p$ a polynomial). 
Let $\cI_I^g$ be the set of interpolators of sub-gaussian growth. % with $I$ omitted for $I = (0,1/2)$. 

For $f \in \cI_I^g$ with $I = (\alpha,\beta) \subset [0,1/2]$, 
we here introduce a sesqui-linear form on $\sN$ as follows. Continuous functions 
\[ 
F(s,t) = \bigl(h(t-i/2)| f_\phi(s-ir) \phi^{it+1/2} {}_\phi g(t-s+ir-i/2)\bigr)
\] 
of $(s,t) \in \R^2$ parametrized by $r \in I$ are of Gaussian decay 
with their absolutely convergent integrals independent of $r \in I$ owing to 
Cauchy's integral theorem. Moreover $F(s,t)$ does not depend on the choice of supporting $\phi$ either.  
% because, for $\omega \geq \phi$,  
% \[ 
% \phi^{i(s-ir)}\omega^{-i(s-ir)} \omega^{it + 1/2} 
% \omega^{-i(t-s+ir - i/2)} \phi^{i(t-s+ir - i/2)} 
% \] 
% is an analytic continuation from $s$ to $s-ir$ of 
% \[ 
% \phi^{is}\omega^{-is} \omega^{it + 1/2} \omega^{-i(t - s - i/2)} \phi^{i(t-s - i/2)} 
% = \phi^{it + 1/2}. 
% \]

Thus a sesqui-linear form $\langle\ |\ \rangle_f$ on $\sN$ is well-defined by 
\begin{multline*}
\langle h|g\rangle_f 
= \int_{\R^2} dsdt\,\bigl(h(t-i/2) | f_\phi(s-ir) \phi^{it + 1/2} {}_\phi g(t-s+ir -i/2)\bigr)\\ 
= \int_{\R^2} dsdt\,\bigl(h(t-i/2) (g^*)_\phi(-t+s+ir -i/2) \phi^{i(s-t)}| f_\phi(s-ir) \phi^{is + 1/2}\bigr)     
\end{multline*}
as far as $r \in I$ and $\phi \in M_*^+$ supports $f$ and $g$, 
which behaves well under the *-operation: 
$\langle g | h\rangle_{f^*} = \overline{\langle h|g\rangle_f}$. 
% In fact, 
% \[ 
% f^*_\phi(z) = (\phi^{i\overline{z}}f(-\overline{z}))^* 
% = (\sigma^\phi_{\overline{z}}(f_\phi(-\overline{z})))^* 
% = \sigma^\phi_z(f_\phi(-\overline{z})^*)
% \]
% for $z = s-ir$ is used to have (with Einstein's convention on integrals) 
% \begin{align*}
% \langle g|h\rangle_{f^*} 
% &= (g(t-i/2) | f^*_\phi(s-ir) \phi^{it+1/2} {}_\phi h(t-s+ir - i/2))\\
% &= \langle {}_\phi g(t-i/2)^* \phi^{-it + 1/2} 
% \sigma^\phi_z(f_\phi(-\overline{z})^*) \phi^{it+1/2} {}_\phi h(t-z-i/2) \rangle\\ 
% &= \langle {}_\phi g(-\overline{(-t-i/2)})^* \phi^{1/2} 
% \sigma^\phi_{z-t}(f_\phi(-\overline{z})^*) \phi^{1/2} {}_\phi h(t-z-i/2) \rangle\\ 
% &\text{(the variable being changed as $t \mapsto t+z$)}\\
% &= \langle {}_\phi g(t + \overline{z} -i/2)^* \phi^{1/2} 
% \sigma^\phi_{-t}(f_\phi(-\overline{z})^*) \phi^{1/2} {}_\phi h(t-i/2) \rangle. 
% \end{align*}
% Thus 
% \begin{align*}
% \overline{\langle g|h\rangle_{f^*}} 
% &= \langle {}_\phi h(t-i/2)^* \phi^{1/2} 
% \sigma^\phi_{-t}(f_\phi(-\overline{z})) \phi^{1/2}  
% {}_\phi g(t + \overline{z} -i/2) \rangle. \\ 
% &= (h(t-i/2) | f_\phi(-\overline{z}) \phi^{it+1/2} {}_\phi g(t+\overline{z} -i/2))\\
% &= (h(t-i/2) | f_\phi(-s-ir) \phi^{it+1/2} {}_\phi g(t+s+ir -i/2))\\
% &= (h(t-i/2) | f_\phi(s-ir) \phi^{it+1/2} {}_\phi g(t-s+ir -i/2))\\
% &= \langle h | g\rangle_f. 
% \end{align*}
Notice that, when $f \in \sN$, $\langle h|g\rangle_f$ is reduced to $(h\tau^{1/2}|fg\tau^{1/2})$.  
% in view of 
% \begin{align*}
% \int f_\phi(s-ir) &\phi^{it+1/2} {}_\phi g(t-s+ir-i/2)\, ds\\
% &= \int f_\phi(s) \phi^{it+1/2} {}_\phi g(t-s-i/2)\, ds\\
% &= \int f(s) g(t-s-i/2)\, ds = (fg)(t-i/2). 
% \end{align*}

We interprete the sequilinear form $\langle\ |\ \rangle_f$ as defining an operator $K$ 
in a kernel form by 
$(h\tau^{1/2}|K(g\tau^{1/2})) = \langle h|g\rangle_f$, which is referred to as 
the \textbf{virtual operator} of $f(z)$ and denoted by $f$ itself. 

Note that the *-operation on interpolators is compatible with the associated virtual operators; 
$(h\tau^{1/2}| f(g\tau^{1/2})) = \overline{(g\tau^{1/2}| f^*(h\tau^{1/2}))}$ for $g,h \in \sN$, and  
virtual operators are affiliated to $N$ in the sense that 
$(hk^*\tau^{1/2}|f(g\tau^{1/2})) = (h\tau^{1/2} | f(gk\tau^{1/2}))$ for $g,h,k \in \sN$.  
% \begin{align*}
% (hk^*&\tau^{1/2}| f(g\tau^{1/2}))\\ 
% &= \lim\int (h(t-u - i/2) k(-u)^* | f_\phi(s-ir) \phi^{it + 1/2} {}_\phi g(t-s+ir-i/2))\\
% &= \lim\int (h(t-u - i/2) | f_\phi(s-ir) \phi^{it + 1/2} {}_\phi g(t-s+ir-i/2)k(-u))\\
% &= \lim\int (h(t - i/2) | f_\phi(s-ir) \phi^{it + iu + 1/2} {}_\phi g(t+u-s+ir-i/2)k(-u))\\
% &= \lim\int (h(t - i/2) | f_\phi(s-ir) \phi^{it - iu + 1/2} {}_\phi g(t-u-s+ir-i/2)k(u))\\
% &= \lim\int (h(t - i/2) | f_\phi(s-ir) \phi^{it + 1/2} {}_\phi(gk)(t-s+ir-i/2))\\
% &= (h\tau^{1/2} | f(gk\tau^{1/2})). 
% \end{align*} 

Let $D(f)$ be the set of vectors $g\tau^{1/2} \in \sN\tau^{1/2}$ which 
makes the conjugate-linear functional $h\tau^{1/2} \mapsto (h\tau^{1/2}|f(g\tau^{1/2}))$ bounded. 
For $g\tau^{1/2} \in D(f)$, if the vector $\xi \in \sH$ satisfying $(h\tau^{1/2}|\xi)$ is denoted by 
$f(g\tau^{1/2})$, then we obtain a linear operator on $\sH$ by 
$D(f) \ni g\tau^{1/2} \mapsto f(g\tau^{1/2}) \in \sH$. 

A virtual operator is said to be densely defined if $D(f)$ is dense in $\sH$. 
When the sesqui-linear form $\langle\ |\ \rangle_f$ itself is bounded, $D(f) = \sN \tau^{1/2}$ and 
the associated linear operator $\sN\tau^{1/2} \to \sH$ is bounded and identified with an element $y \in N$ 
in such a way that $\langle h|g\rangle_f = \bigl(h\tau^{1/2} | y(g\tau^{1/2})\bigr)$ for $g,h \in \sN$. 

We next introduce the virtual vector as a conjugate-linear form on $\sN^2 \tau^{1/2}$. 

\begin{Lemma}\label{OA}%one-analytic
If $\phi \in M_*^+$ supports $g,h \in \sN$, then vector-valued functions $(hg^*)_\phi(s) \phi^{1/2}$ 
and $\phi^{1/2} {}_\phi(hg^*)(s)$ of $s \in \R$ 
are analytically continued to $L^2(M)$-valued norm-continuous functions 
$(hg^*)_\phi(z) \phi^{1/2}$ and $\phi^{1/2} {}_\phi(hg^*)(z)$ of $z \in \R - i[0,1]$ so that these 
are of Gaussian decay and, for $0 \leq r \leq 1/2$, satisfy 
\begin{align*} 
(hg^*)_\phi(s-i(1-r)) \phi^{1/2} 
&= \int_\R h(t-i/2) (g^*)_\phi(-t+s + ir -i/2) \phi^{-it}\, dt,\\  
\phi^{1/2} {}_\phi(hg^*)(s-i(1-r))
&= \int_\R \phi^{it} {}_\phi h(t+s + ir -i/2) g^*(-t -i/2)\, dt 
\end{align*}
respectively. 
\end{Lemma}

\begin{proof}
We already know that $(hg^*)_\phi(s)$ has an s*-continuous analytic extension $(hg^*)_\phi(z) \in M$ 
to $z \in \R - i[0,1/2]$ so that $(hg^*)_\phi(s -i/2) \phi^{1/2} = f(s-i/2)\phi^{-is}$, whereas 
\begin{align*} 
(hg^*)_\phi(s -i/2) \phi^{1/2} %=  \int h(t-i/2) g^*(-t + s)\, dt 
% &= \int h(t) g^*_\phi(-t + s -i/2) \phi^{-it + 1/2}\, dt \\
% &= \int h_\phi(t) \phi^{is + 1/2} {}_\phi g^*(-t + s -i/2) \phi^{-is}\, dt \\
% &= \int h_\phi(t-i/2) \phi^{is + 1/2} {}_\phi g^*(-t + s) \phi^{-is}\, dt \\
% &= \int h(t-i/2) \phi^{i(s-t)} {}_\phi g^*(-t + s) \phi^{-is}\, dt \\
% &= \int h(t-i/2) g^*(-t + s) \phi^{-is}\, dt \\
&= \int h(t-i/2) g^*_\phi(-t + s) \phi^{-it}\, dt \\
\end{align*}
is analytically continued to the norm-continuous function 
\[ 
\int h(t-i/2) g^*_\phi(-t + z)\phi^{-it}\, dt 
\] 
of $z \in \R -i[0,1/2]$, which is of Gaussian decay as a convolution of functions of Gaussian decay.  
\end{proof}

The sesqui-linear form $\langle h|g\rangle_f$ is now expressed by
\[ 
\langle h|g\rangle_f 
%&= \int dsdt\,\bigl(h(t-i/2) (g^*)_\phi(-t+s+ir -i/2) \phi^{i(s-t)}| f_\phi(s-ir) \phi^{is + 1/2}\bigr)\\
= \int_\R ds\,\bigl( (hg^*)_\phi(s-i(1-r)) \phi^{is + 1/2} | f_\phi(s-ir) \phi^{is + 1/2}\bigr), 
\] 
whenever $0 < r < 1$ and $\phi$ supports $g$, $h$ as well as $f$, which reveals that 
a conjugate-linear form $f\tau^{1/2}$ on $\sN^2\tau^{1/2}$ is well-defined by the relation
\[ 
(hg^*\tau^{1/2}|f\tau^{1/2}) = \langle h|g\rangle_f
\] 
and called the \textbf{vitual vector} of $f$. 

%Here the inner product notation is borrowed to stand for conjugate-linear forms. 
Note that the virtual vector of $f^*$ is given by $(f\tau^{1/2})^*$ which is defined by 
$(\xi| (f\tau^{1/2})^*) = \overline{(\xi| f\tau^{1/2})}$ for $\xi \in \sN^2 \tau^{1/2}$: 
\[ 
\langle h|g\rangle_{f^*} = \overline{\langle g| h\rangle_f} 
= \overline{(gh^*\tau^{1/2}|f\tau^{1/2})} = (hg^*\tau^{1/2}| (f\tau^{1/2})^*). 
\] 

These are also referred to as a \textbf{boundary operator} and a \textbf{boundary vector} for 
$I = (0,\nu)$ and $I = (\nu,1/2)$ with additional notations 
$\int f(t)\, dt$ and $\oint f(t-i/2)\, dt$ respectively. We now focus on these. 

\medskip
\noindent\textbf{Boundary Operator:} 
In extracting linear operators from the kernel form of boundary operators, 
the following illustrates the meaning of boundary (limit). 

Let $D$ be the set of s*-continuous sections of $\{ M(it+1/2)\}$ of Gaussian decay, 
which is a topological vector space of inductive limit 
of Banach spaces $D_\delta = \{ \{ \xi(t)\} \in \{ M(it+1/2)\}; \| \xi\|_\delta < \infty\}$ 
with $\| \xi\|_\delta = \sup\{ e^{\delta t^2} \| \xi(t)\|; t \in \R\}$. 
The embedding $D_\delta \to \sH$ is norm-continuous and therefore 
so is $D \to \sH$. 
For $f \in \cI_I^g$ with $I = (0,\nu)$ and $\xi \in D_\delta$
\[ 
\int_\R f_\varphi(s-ir) \varphi^{it}\xi\, ds 
\]
is norm-convergent in $D_{\delta'}$ for any $\delta' < \delta$ and 
gives a bounded linear map $D_\delta \to D_{\delta'}$, 
which depends continuously on $r \in I$ in the norm-topology of $\sB(D_\delta,D_{\delta'})$. 
The induced continuous linear operator 
on $D$ is then denoted by $\int_\R f_\varphi(s-ir) \varphi^{is}\, ds$. 
We say that $\int_\R f_\varphi(s-ir) \varphi^{is}\, ds$ is bounded 
if it is bounded as a densely defined linear operator on $\sH$. 

Note that, if $\int_\R f_\varphi(s-ir)\varphi^{is}\, ds \in \sB(\sH)$ is locally norm-bounded for $r \in I$, 
it is s-continuous in $r \in I$ by the density of $D$ in $\sH$. 

\begin{Lemma}
Let $f \in \cI_I^g$ be supported by $\varphi \in M_*^+$. 
Assume that 
\[ 
D \ni \xi \mapsto \int_\R f_\varphi(s-ir) \varphi^{is}\xi\, ds \in \sH 
\]
gives rise to a bounded linear operator $y_r = \int_\R f_\varphi(s-ir) \varphi^{is}$ on $\sH$ and 
\[ 
y = \int_\R f(s-i0)\, ds = \lim_{r \to +0} \int_\R f_\varphi(s-ir) \varphi^{is}\, ds 
\]
exists in the w*-topology of $N$. 

Then the boundary operator of $f(z)$ is bounded and given by the above limit. 
\end{Lemma} 

\begin{proof}
Given $g \in \sN$ and $\varphi \in M_*^+$, 
choose $\phi \in M_*^+$ so that it supports $g$ and majorizes $\varphi$. 
Then, 
\[ 
\R^2 \ni (s,t) \mapsto 
\varphi^{is} f(t-s-i/2) \phi^{-it} 
\in L^2(M)
\] 
is analytically extended to an $L^2(M)$-valued norm-continuous function 
$(\varphi^{iz} \phi^{-iz}) \bigl(\phi^{iz} g(t-z-i/2)\bigr)$ of $z \in \R - i[0,1/2]$ and $t \in \R$, 
which is denoted by $\varphi^{iz} g(t-z-i/2)$. 

Since $\varphi^{iz} g(t-z-i/2) = (\varphi^{iz} \phi^{-iz}) \phi^{it + 1/2} {}_\phi g(t-z-i/2)$, 
$(\varphi^{iz}g)(t-i/2) = \varphi^{iz} g(t-z-i/2)$ belongs to $D_\delta$ as a function of $t \in \R$ if 
$|g|_\phi(t) = O(e^{-\delta t^2})$. Thus, $\xi_r(t) = \varphi^r g(t+ir - i/2)$
is a $D_\delta$-valued norm-analytic function of $r$. 

By our assumptions, 
s*-continuous family $\{ y_r \}_{r \in I}$ in $N$
converges to $y$ in w*-topology as $r \to +0$, 
whence the operator norm $\| y_r\|$ is bounded in a neighborhood of $r=0$ and we see that 
\[ 
y_r \xi_r = \lim_{r \to 0} y_r\xi_r 
= \lim_{(r',r'') \to (0,0)} y_{r'} \xi_{r''} =  \lim_{r' \to 0} y_{r'} \xi_0 
= y\xi_0.  
\]

Now the identity
\[ 
\int f_\phi(s-ir) \phi^{it+1/2} {}_\phi g(t-s+ir - i/2)\, ds = (y_r\xi_r)(t) 
\]
% \begin{align*} 
% \int f_\phi(s-ir) \phi^{it+1/2} &{}_\phi g(t-s+ir - i/2)\, ds\\ 
% &= \int f_\phi(s-ir) \phi^{i(s-ir)} g(t-s+ir - i/2)\,ds\\ 
% &= \int f_\varphi(s-ir) \varphi^{i(s-ir)} g(t-s+ir - i/2)\, ds\\ 
% &= (y_r\xi_r)(t) 
% \end{align*}
is used to get 
\[ 
\langle h|g\rangle_f = (h\tau^{1/2}|y_r\xi_r) = (h\tau^{1/2}|y\xi_0) = 
(h\tau^{1/2}|y(g\tau^{1/2})). 
\]
\end{proof}

\begin{Corollary} 
Let $f(z)$ be an interpolator on $I = (0,\nu)$ ($0 < \nu \leq 1/2$) and 
suppose that $f$ is supported by a $\phi \in M_*^+$ so that 
$f_\phi(z) = f(z)\phi^{-iz}$ is a scalar operator of polynomial growth 
with its horizontal Fourier transform 
$\int_\R f_\phi(s-ir) e^{is\lambda}\, ds$ being in $L^\infty(\R)$ for a small $r>0$ and w*-converging to
$\widehat{f_\phi} \in L^\infty(\R)$ as $r \to 0$, then the boundary operator of $f(z)$ is a bounded operator 
\[ 
\widehat{f_\phi}(\log \phi) = \int_\R \widehat{f_\phi}(\lambda) E(d\lambda) \in N. 
\] 
Here $E(\cdot)$ denotes the spectral measure of $\phi^{it}$: $\phi^{it} = \int_\R e^{it\lambda} E(d\lambda)$. 
\end{Corollary}

\begin{proof}
Due to the left trivialization $[\phi]L^2(N) \cong L^2(\R)\otimes [\phi] L^2(M)$, the whole thing is 
reduced to $L^\infty(\R)$ on $L^2(\R)$ and the classical harmonic analysis on the real line works. 
\end{proof}

\begin{Example}
If $f(z) \phi^{-iz}$ extends to a bounded w*-continuous $M$-valued function of $z \in \R - i[0,\nu)$ 
in such a way that there exists an integrable function $\rho(t)$ satisfying 
$\|f(t-ir)\phi^{-i(t-ir)}\| \leq \rho(t)$ for $t \in \R$ and $0 \leq r < \nu$, 
then the boundary operator is bounded and hence belongs to $N$. 
\end{Example}

\begin{Example}
For $\phi \in M_*^+$ and $\mu \in \C$, 
consider an interpolator $f(z) = \frac{1}{\mu + iz} \phi^{iz}$ on $I$ with $I$ specified 
according to $\mu$ as follows: 
\begin{enumerate}
\item $I = (0,1/2)$ ($\Re\mu \geq 0$). 
The boundary operator is given by $2\pi(1\vee \phi)^{-\mu}$. 
\item 
Either $I = (0,-\Re\mu)$ ($-1/2 < \Re\mu < 0$) or 
$I = (0,1/2)$ ($\Re \leq -1/2$). 
Then the boundary operator is given by
$-2\pi (1\wedge \phi)^{-\mu}$ for $\Re\mu < 0$. 
\end{enumerate}

Here, with the help of a spectral decomposition $\phi^{it} = \int_\R e^{it\lambda} E(d\lambda)$, 
\[ 
(1\vee \phi)^{-\mu} = \int_0^\infty e^{-\mu \lambda} E(d\lambda), 
\quad 
(1\wedge \phi)^{-\mu} = \int_{-\infty}^0 e^{-\mu \lambda} E(d\lambda). 
\]
\end{Example}

\medskip
\noindent\textbf{Boundary Vector:} 
We next look into boundary vectors. 
% As already remarked, the boundary vector of $f \in \sN$ 
% is equal to $f\tau^{1/2} \in \sH$ by the continuity including the boundary. 
Let $f(z) \in \sI_I^g$ with $I = (\nu,1/2)$ and $g,h \in \sN$. 
In the expression 
\[
(hg^*\tau^{1/2} | f\tau^{1/2}) 
= \int_\R ds\,\bigl( (hg^*)_\varphi(s-i(1-r)) \varphi^{is + 1/2} | f_\varphi(s-ir) \varphi^{is + 1/2}\bigr)  
\] 
($g$, $h$ and $f(z)$ being supported by $\varphi \in M_*^+$), 
notice that the norm-convergence 
$\displaystyle \lim_{r \to 1/2} (hg^*)_\varphi(s-i(1-r)) \varphi^{1/2} = (hg^*)(s-i/2) \varphi^{-is}$ 
in $L^2(M)$ is uniformly in $s \in \R$ and 
the domination $\| (hg^*)_\varphi(s-i(1-r)) \varphi^{1/2}\| \leq C e^{-\delta s^2}$ holds uniformly in $r$, 
whereas $\|f_\varphi(s-ir)\| = O(e^{\epsilon s^2})$ uniformly in $r$ for any $\epsilon>0$. 

Thus, if $f(z)$ satisfies the condition that 
\begin{enumerate} 
\item 
$\rho_\varphi(s) = \sup\{ \| f_\varphi(s-ir)\|; r \in (\nu,1/2)\}$ 
is a locally integrable function of $s \in \R$ for some supporting $\varphi$ 
and 
\item 
we can find a locally integrable measurable section $\eta(s) \in M(is+1/2)$ 
so that, for a sufficiently large $\phi$ and for almost all $s$, 
$f_\phi(s-ir) \phi^{is + 1/2}$ converges weakly to $\eta(s)$ in $M(is+1/2)$ as $r \to 1/2$, 
\end{enumerate}
then we have the expression 
\[ 
(hg^*\tau^{1/2} | f\tau^{1/2}) = \int_\R \bigl( (hg^*)(s-i/2) | \eta(s) \bigr)\, ds, 
\] 
% \begin{align*} 
% (&hg^*\tau^{1/2} | f\tau^{1/2})\\ 
% &= \lim_{r \to 1/2} \sum_j \int_\R ds\, 
% \bigl( (hg^*)_\phi(s-i(1-r)) \phi^{is+1/2} | \delta_j\bigr) 
% \bigl( \delta_j | f_\phi(s-ir) \phi^{is+1/2} \bigr)\\
% &= \sum_j \int_\R ds\, 
% \bigl( (hg^*)(s-i/2) | \delta_j\bigr) 
% ( \delta_j | \eta(s) ) 
% = \int_\R \bigl( (hg^*)(s-i/2) | \eta(s) \bigr)\, ds,  
% \end{align*}
which shows that the boundary vector of $f(z)$ is represented 
by the measurable section $\eta(s) \in M(is+1/2)$. 
Note here that 
% $\rho_\phi(s) \leq \rho_\varphi(s) \sup\{ \| \varphi^r \phi^{-r}\|; 0 \leq r \leq 1/2 \}$ if
% $\phi \geq \varphi$ and 
$\| \eta(s)\|$ is of sub-gaussian growth. 

\begin{Example}
If $f_\phi(z)$ is extended to an $M$-valued w*-continuous function of 
$z \in \R - i(\nu,1/2]$, then $\rho_\phi(s)$ is locally bounded and 
$\eta(s) = f_\phi(s-i/2) \phi^{is+1/2} = f(s-i/2)$ meets the requirements. 
\end{Example}

% A similar machinery works 
% if $f_\phi(t-ir) \phi^{it+1/2} \in M(it+1/2)$ is square-integrable and 
% \[ 
% \lim_{r \to 1/2} \oint f_\phi(s-ir) \phi^{is + 1/2}\, ds = 
% \oint_\R \eta(s)\, ds 
% \]
% in norm topology of $\sH$; $\eta \in \sH$ is the boundary vector of $f$. 

% The rewriting 
% \begin{align*} 
% \langle h|g\rangle_f &= \lim_{r \to 1/2} 
% \iint dsdt\,(h(t-i/2) (g^*)_\phi(-t+s+ir -i/2) \phi^{i(s-t)} |\\
% &\hspace{7cm}f_\phi(s-ir) \phi^{is + 1/2})\\ 
% &= \iint dsdt\,(h(t-i/2) (g^*)_\phi(-t+s) \phi^{i(s-t)} | \eta(s))\\ 
% &= \iint dsdt\,(h(t-i/2) g^*(-t+s) | \eta(s))\\ 
% &= \int (hg^*(s-i/2) | \eta(s))\,ds = (hg^* \tau^{1/2} | \oint \eta(s)\, ds)  
% \end{align*}
% shows that $\langle h|g\rangle_f$ induces a bounded conjugate-linear form on $\sN^2 \tau^{1/2}$. 

% If $f_\phi(z) = f(z)\phi^{-iz}$ is uniformly bounded on $\R -i(\nu,1/2)$ and extended to 
% a w*-continuous $M$-valued function of $z \in \R - i(\nu,1/2]$, then the boundary 
% vector exists and is given by 
% \[ 
% \sN^2 \tau^{1/2} \ni h\tau^{1/2} \mapsto 
% \int_\R \bigl(h(t-i/2) | f_\phi(t-i/2) \phi^{it + 1/2}\bigr)\, dt. 
% \] 
% Here the integral is independent of the choice of supporting $\phi$ and reduced to 
% $(h\tau^{1/2} | f\tau^{1/2})$ for $f \in \sN$. 

\begin{Example}
Again consider $f(z) = \frac{1}{\mu + iz} \varphi^{iz}$ on $I$ 
but this time $I = (-\Re\mu,1/2)$ if $-1/2 < \Re\mu < 0$ and $I = (0,1/2)$ otherwise. 

Then, for $\Re\mu \not= -1/2$, 
the boundary vector of $f$ belongs to $\sH$ and 
is given by $f(t-i/2) = (\mu + it + 1/2)^{-1} \varphi^{it+1/2}$.

When $\Re\mu \not\in [-1,-1/2]$, the expression 
\begin{align*}
(k\tau^{1/2}| f\tau^{1/2}) 
%&= \int_\R  \frac{1}{it + \mu + 1/2} (k(t-i/2)|\varphi^{it+1/2})\, dt\\ 
&= \int_\R  \frac{1}{it + \mu + 1/2} \langle k^*(-t-i/2)\varphi^{it+1/2}\rangle\, dt\\ 
\end{align*}
for $k \in \sN^2$ is analytically changed in the integration variable to get 
% \[ 
% \int_\R  \frac{1}{it + \mu + 1} \langle k^*(-t)\phi^{it+1}\rangle\, dt
% = \int_\R  \frac{\phi(k^*(-t)\phi^{it})}{it + \mu + 1} \, dt, 
% \] 
% to get 
\[ 
(k\tau^{1/2}| f\tau^{1/2}) = \int_\R  \frac{\phi(k^*(-t)\varphi^{it})}{it + \mu + 1} \, dt.  
\] 
% for $\Re\mu \not\in [-1,-1/2]$ and, if $k$ is supported by $\phi$,  
% \[
% (k\tau^{1/2}| f\tau^{1/2}) = \int_\R  \frac{\phi(k^*(-t)\phi^{it})}{it + \mu + 1} \, dt 
% -2\pi \phi\bigl(k^*_\phi(-i\mu-i)\bigr) 
% \] 
% for $-1 < \Re\mu < -1/2$. 

Thus the parametric limit of $f\tau^{1/2}$ exists in simple convergence 
as $\mu$ approaches to a point in $\Re\mu = -1/2$ 
from the right ($\Re\mu > -1/2$). 

Now let $\mu = im -1/2$ ($m \in \R$) be on the critical line $\Re\mu = -1/2$ and 
set $\epsilon = 1/2 - r$. 
By Lemma~\ref{OA}, we have 
\begin{align*}
\langle h|g \rangle_f 
% &= \int_\R \frac{1}{i(s+m) - \epsilon} \phi\bigl( {}_\phi(gh^*)(-s-i\epsilon-i/2) \bigr)\,ds \\
% &= \int_\R \frac{1}{i(s+m) + 1/2} \phi\bigl( {}_\phi(gh^*)(-s) \bigr)\,ds \\
&= \int_\R \frac{1}{i(s+m) + 1/2} \phi\bigl( \phi^{is} (gh^*)(-s) \bigr)\,ds,  
\end{align*}
which reveals that the boundary vector of $f(z)$ coincides with 
\[ 
\lim_{\epsilon \to +0} \oint_\R \frac{1}{i(t+m) + \epsilon} \phi^{it+1/2}\, dt.  
\] 
\end{Example}

% \begin{Remark}
% An intuitive explanation of the above coincidence is as follows:
% Taking the boundary limit $t-ir \to t - i/2$ is `equivalent to' 
% moving $\mu = im + \epsilon$ by $\epsilon \to +0$ and therfore 
% the formula 
% \[
% \langle h|g\rangle_f = \int_\R \frac{1}{it + \mu + 1} \phi\bigl( \phi^{it} (gh^*)(-t) \bigr)\,dt 
% \]
% ramains valid for $\Re \mu + 1 >0$. 
% Particularly, this is the case for $\Re\mu = -1/2$ and the above expression in fact gives 
% the boundary vector $f\tau^{1/2}$ of $f$ as a conjugate-linear form on $\sN^2 \tau^{1/2}$. 
% \end{Remark}

We now generalize the notion of interpolators on $I = (0,1/2)$ so that $f(z)$ is allowed to 
be not defined on a compact subset $K$ of $\R - i(0,1/2)$. 
The various analyticity is then defined just avoiding $K$. Since the growth condition is 
about horizontal asymptotics, it remains having a meaning as well. 

We introduce the residue operator $R_f = \oint_K f(z)\, dz: \sN \tau^{1/2} \to \sH$ by 
\[ 
R_f (g\tau^{1/2}) = \oint_\R \left(\oint f(z) g(t-z-i/2)\, dz \right)\, dt. 
\]
Here $f(z) g(t-z-i/2) = f_\phi(z) \phi^{it+1/2} {}_\phi g(t-z-i/2)$ 
is an $M(it+1/2)$-valued analytic function of $z \in (\R -i[0,1/2]) \setminus K$ 
and the coutour integral is performed by surrounding $K$.

\begin{Theorem}[Trace Formula]
Let $f(z)$ be an interpolator on $(0,1/2)$ of sub-gaussian growth 
and assume that the boundary vector 
$f\tau^{1/2} = \oint_\R f(t-i/2)\, dt$ exists in $\sH$. 

Then the sum of the boundary operator $f$ and the residue operator $R_f$ is $\tau$-measurable and
we have 
\[ 
\tau((f+R_f)^*(f+R_f)) = (f\tau^{1/2} | f\tau^{1/2}) 
= \int_\R \bigl(f(t-i/2)|f(t-i/2)\bigr)\, dt. 
\]
\end{Theorem}

\begin{proof} 
Let $V_f$ be the virtual operator of $f(z)$ ($z \in \R - i(1/2 - \epsilon,1/2)$). 
By the residue formula, $V_f = f + R_f$ and, for $g,h \in \sN$, 
\[ 
(h\tau^{1/2} | V_f(g\tau^{1/2})) %= (hg^*\tau^{1/2}| f\tau^{1/2})
= (h\tau^{1/2} | (f\tau^{1/2})g) = (h^{1/2}\tau^{1/2} | l(f\tau^{1/2}) (g\tau^{1/2}))
\]
shows that the virtual operator $V_f$ is closable 
with its closure given by $l(f\tau^{1/2})$. Lemma~\ref{HS} is then the applied to get the assertion. 
\end{proof} 

\begin{Corollary}
If $f(z)$ is analytic on the whole $\R - i(0,1/2)$ additionally, then 
the boundary operator $f$ is $\tau$-measurable and we have 
\[ 
\tau(f^*f)) = (f\tau^{1/2} | f\tau^{1/2}) = \int_\R \bigl(f(t-i/2)|f(t-i/2)\bigr)\, dt. 
\] 
\end{Corollary} 

\begin{Example}
Let $G \in L^2(\R)$ and suppose that 
its Fourier transform $\widehat G(\lambda) = \int_\R G(t) e^{-i\lambda t}\, dt$ is integrable and satisfies 
$\int_0^\infty |\widehat G(\lambda)|^2 e^{\lambda}\, d\lambda < \infty$. 

Then the inverse Fourier transform $G_w$ of $\widehat G(\lambda) e^{iw\lambda}$ belongs to 
$L^2(\R) \cap C_0(\R)$ and 
depends on $w \in \R -i[0,1/2]$ norm-continuously for both $\|\cdot\|_\infty$ and $\|\cdot\|_2$. 
% Here are relevant estimates. 
% \begin{align*} 
% \int_\mu^\infty |\widehat G(\lambda) (e^{iz\lambda} - e^{iw\lambda})|^2\, d\lambda 
% &\leq 4\int_\mu^\infty |\widehat G(\lambda)|^2 e^\lambda\, d\lambda,\\ 
% \left( \int_{\mu}^\infty |\widehat G(\lambda) e^{iz\lambda}|\, d\lambda \right)^2
% &\leq \int_\mu^\infty |\widehat G(\lambda)|^2 e^\lambda\, d\lambda \int_\mu^\infty e^{-\lambda}\, d\lambda,\\ 
% \int_{-\infty}^{-\mu} |\widehat G(\lambda) e^{iz\lambda}|\, d\lambda 
% &\leq \int_{-\infty}^\mu |\widehat G(\lambda)|\, d\lambda. 
% \end{align*} 
Since, for $F \in L^2(\R)$, 
\[ 
(F|G_w) = \frac{1}{2\pi} \int_\R \overline{\widehat F(\lambda)} {\widehat G}(\lambda)e^{iw\lambda}\, d\lambda
\] 
is analytic in $w$ and $G_s$ is reduced to the translation $G(t+s)$ of $G(t)$, 
$F(t)$ is analytically extended to $F(z)$ so that 
$G_w(t) = G(t+w)$ for $w \in \R - i[0,1/2]$ and $t \in \R$. 

Now, for $\phi \in M_*^+$, $g(z) = G(z) \phi^{it}$ defines an interpolator on $(0,1/2)$ which vanishes 
at $\Re z = \pm\infty$. 
Since $\phi^{it}$ on $\sH$ is given by translation on $L^2(\R)\otimes [\phi]L^2(M)[\phi] 
\cong [\phi]\sH[\phi]$, the associated boundary operator is bounded and the boundary vector is given by 
$\oint_\R G(t-i/2) \phi^{it + 1/2}\, dt$ so that
\begin{align*} 
\tau(g^*fg) &= \int_{\R^2} (g(t-i/2)|f(s) g(t-s-i/2))\, dsdt\\ 
&= \phi(1) \int_{\R^2} \overline{G(t-i/2)} F(s) G(t-s-i/2)\, dsdt\\ 
&= \frac{\phi(1)}{2\pi} \int_\R \widehat F(\lambda) |\widehat G(\lambda)|^2 e^\lambda\, d\lambda. 
\end{align*} 
Here, for $F \in L^1(\R)$, an $L^1$-section $\{ f(t)\} $ of $\{ M(it)\}$ is defined by 
$f(t) = F(t) \phi^{it}$ and $f = \int_\R f(t)\, dt \in N$. 

Thus, letting $A$ be the W*-subalgebra of $[\phi]N[\phi]$ generated by $\{ \phi^{it}; t \in \R\}$, 
$L^2(A,\tau)$ is identified with $L^2(\R, e^\lambda\,d\lambda)$ by a unitary map
\[ 
U_\phi: L^2(A,\tau) \ni g\tau^{1/2}) \mapsto \sqrt{\frac{\phi(1)}{2\pi}} \widehat G(\lambda) 
\in L^2(\R, e^\lambda\,d\lambda)
\] 
so that $\phi^{is}$ on $L^2(A,\tau)$ is realized by a multiplication of 
the function $e^{-is\lambda}$ of $\lambda \in \R$. 
\end{Example}

\begin{Example}\label{negative}
For $-1/2 < \Re\beta < 0$ and $\phi \in M_*^+$, the interpolator $f(z) = \frac{1}{\beta + iz} \phi^{iz}$ 
has $-2\pi(1\wedge \phi)^{-\beta} \in N$ as the boundary operator. The residue operator is calculated 
by the realization $L^\infty(A)$ on $L^2(A)$ as 
\[ 
\int_{|z-i\beta| = \epsilon} \frac{1}{\beta + iz} e^{i\lambda z}\, dz = 2\pi e^{-\beta \lambda}, 
\] 
which is therefore $2\pi \phi^{-\beta}$. Adding these, we see that $(1\vee \phi)^{-\beta}$ is 
in the Hilbert-Schmidt class and hence, for $x \in M$ and $\mu = -r + is \in - (0,1) + i\R$, 
$x(1\vee \phi)^{-\mu} = x(1\vee \phi)^{r/2-is} (1\vee \phi)^{r/2}$ is in the trace class with 
\begin{align*} 
2\pi \tau(x(1\vee \phi)^{-\mu}) 
&= \phi(x) \int_\R \frac{1}{-it + (1-r)/2} \frac{1}{i(t+s) + (1-r)/2}\, dt\\ 
&= \frac{\phi(x)}{is - r + 1} = \frac{\phi(x)}{\mu + 1}. 
\end{align*}  
\end{Example}

% Let $\omega = \sum \omega_i$ with $\{ \omega_i\}$ an orthogonal family in $M_*^+$ be a weight on $M$. 
% In terms of a spectral decomposition $\omega^{it} = \displaystyle \int_\R e^{it\tau} E(d\tau)$ in $N$, 
% we introduce a multiplicative family $\{ (1\vee \omega)^{-\mu}\}_{\mu \in \C}$ of 
% densely defined closed operators affiliated with $N$ by 
% \[ 
% (1\vee \omega)^{-\mu} = \int_0^\infty e^{-\mu s} E(ds), 
% \]
% which satisfies 
% $\bigl( (1\vee \omega)^{-\mu}\bigr)^* = (1\vee\omega)^{-\overline{\mu}}$ and consists of 
% bounded operators when restricted to $\Re \mu \geq 0$. 

% Note that there is no need for worrying about integral boundaries
% because $E(\cdot)$ has no point spectrum. 

% If $\mu>0$, $(1\vee\omega)^{-\mu}$ is positive, decreasing in $\mu$ and converges to 
% the support projection $[1\vee\omega]$ of $(1\vee\omega)^{-\mu}$ as $\mu \to 0$ 
% in strong operator topology. 
% Moreover, for $\Re \mu > 0$, we have also the integral expression 
% \[ 
% (1\vee \omega)^{-\mu} = \frac{1}{2\pi} \int_\R \frac{1}{\mu + it} \omega^{it}\, dt.  
% \]
% (The right hand side has a meaning that it is weakly convergent for sufficiently many vectors in $\sH$.)

Although Haagerup deals only with the case $\mu=0$ and its scaled variation, the following generalization 
should also be attributed to him. 

\begin{Theorem}[Haagerup's Trace Formula] 
Let $\omega$ be a weight on $M$ in our sense. 
The trace of a positive operator $(1\vee \omega)^{-\mu}$ with $\mu \in \R$, 
which belongs to $N$ for $\mu \geq 0$ and 
is affiliated to $N$ for $\mu < 0$, is given by 
\[ 
\tau( (1\vee \omega)^{-\mu} ) 
= \begin{cases} 
\frac{\omega(1)}{2\pi(\mu + 1)} &\text{if $\mu > -1$,}\\
\infty &\text{otherwise.} 
\end{cases}
\] 

Moreover, when $\omega \in M_*^+$, for any $x \in M$ and $\mu \in (-1,\infty) + i\R$, 
the $\tau$-measurable operator $x(1\vee\omega)^{-\mu}$ is in the trace class 
and we have 
\[ 
\tau(x(1\vee \omega)^{-\mu}) = \frac{\omega(x)}{2\pi(\mu + 1)}. 
\] 
\end{Theorem}

\begin{proof}
Assume $\omega \in M_*^+$. Then $\omega^{it}$ is realized as a multiplication operator 
on $L^2(\R,e^\lambda\,d\lambda)$ by a function $e^{-it\lambda}$ of $\lambda \in \R$. 
Consequently $(1\vee \omega)^{-\mu}$ is represented by the function 
$1_{(-\infty,0]}(\lambda) e^{\lambda\mu}$ of $\lambda$, which is integrable relative to 
the measure $e^\lambda\,d\lambda$ if and only if $\Re\mu > -1$ with 
\[ 
\int_{-\infty}^0 e^{\lambda\mu} e^\lambda\, d\lambda = \frac{1}{\mu + 1}. 
\] 

Since our weights are orthogonal sums of elements in $M_*^+$, the formula for $\omega \in M_*^+$ 
remains valid for weights. 

The remaining part is already covered in Example~\ref{negative}. 
\end{proof}

\begin{Remark}~ 
\begin{enumerate}
\item
By the integral expression 
$\displaystyle \int_\R \frac{1}{\mu+it} \omega^{it}\, dt$ of $2\pi(1\vee \omega)^{-\mu}$, 
the formula coincides with the one obtained from the formal argument. 
\item 
The normalization of our trace is different from that in \cite{H} and \cite{Terp} 
by a factor $2\pi$. 
\end{enumerate}
\end{Remark}

Thus, for $\omega \in M_*^+$, the analytic generator $h$ of $\omega^{it}$ as a positive operator on 
$\sH$, which satisfies $\theta_s(h) = e^{-s}h$ (called relative invariance of degree $-1$), 
is $\tau$-measurable in the sense that $\lim_{r \to \infty} \tau([r\vee h]) = 0$. 
Haagerup's ingeneous observation is that the whole $L^p(M)$'s are captured 
as measurable operators on $\sH$ satisfying relative invariance of degree $-1/p$. 

We now go into the reverse problem of characterizing $\tau$-measurable positive operators satisfying 
relative invariance of degree $-1$, which is the heart of Haagerup's correspondence. 

Recall the original approach to this problem:
First establish a one-to-one correspondence between normal weights on $M$ and 
$\theta$-invariant normal weights on $N$. Second the latter is then paraphrased into 
positive operators of relative invariance of degree $-1$ 
by taking Radon-Nikodym derivative with respect to $\tau$. 
Finally, positive operators associated to $M_*^+$ are characterised as $\tau$-measurable 
operators among these. 

Formally the whole processes look natural and seem harmless 
but it is in fact supported by clever and effective controls over infinities 
based on extended positive parts. 

We shall here present an inelegant but down-to-earth proof by continuing elementary Fourier calculus.

\section{Haagerup Correspondence}
Let $h \geq 0$ be a $\tau$-measurable operator on $\sH$ satisfying $\theta_s(h) = e^{-s}h$ for $s \in \R$. 
Our first task here is to identify $h^{it}$ with $\varphi^{it}$ for some $\varphi \in M_*^+$.  

Let $e = [1 \vee h]$ be the support projection of $1\vee h$. 
By the relative invariance of $h$, $\theta_s(e)$ is the support projection of 
$e^s\vee h$ and we have a Stieltjes integral representation of $h$ 
\[ 
h = - \int_{-\infty}^\infty e^s\, d\theta_s(e) = \int_{-\infty}^\infty e^{-s} d\theta_{-s}(e) 
\] 
and set 
\[ 
(1\vee h)^{-\mu} = - \int_0^\infty e^{-\mu s}\, d\theta_s(e),  
\]
which is $\tau$-measurable for any $\mu \in \C$ in view of $\tau(e) < \infty$. 
Notice that $\theta_s(e)$ is continuous in $s \in \R$ and $d\theta_s$ has no spectral jumps. 

Let $x \in M$ and start with the computation
\begin{align*} 
\tau( hx (1\vee h)^{-\mu} ) 
&= \tau( x (1\vee h)^{-\mu}h )
= \tau( x (1\vee h)^{1-\mu} ) \\ 
&= - \int_0^\infty e^{(1-\mu)s} d\tau( x\theta_s(e) ) 
= - \int_0^\infty e^{(1-\mu)s} d(e^{-s}) \tau(xe)\\ 
&= \tau(xe) \int_0^\infty e^{-\mu s} ds
= \frac{1}{\mu} \tau(xe),  
\end{align*}
which is valid for $\Re \mu > 0$. 
% and compatible with 
% \begin{align*}
% \int_\R \theta_s(x(1\vee h)^{-\mu})\, ds 
% &= \frac{1}{2\pi} \int_{\R^2} \frac{1}{\mu + it} x\theta_s(h^{it})\, dsdt\\
% &= \frac{1}{2\pi} \int_{\R^2} \frac{1}{\mu + it} e^{-ist} xh^{it}\, dsdt\\ 
% &= \int_\R \frac{1}{\mu + it} xh^{it} \delta(t)\, dt = \frac{1}{\mu} x[h]. 
% \end{align*}

For $t \in \R$, $\sigma_t(x) = h^{it}xh^{-it}$ ($x \in M$) defines an automorphic action of $\R$ on $M$ 
because $h^{it}xh^{-it}$ is $\theta$-invariant in view of $\theta_s(h^{it}) = e^{-ist} h^{it}$. 
We claim that $\varphi(x) = 2\pi \tau(xe)$ satisfies the KMS-condition for 
the automorphic action $\sigma_t$. 

First notice that $[h] = [\varphi]$. 
In fact, from the definition of $\varphi$ 
and the faithfulness of the standard trace, $(1-[\varphi])e = 0$, which means that 
$e \leq [\varphi]$ and then $[h] = \lim_{s \to -\infty} \theta_s(e) \leq \theta_s([\varphi]) = [\varphi]$. 
Conversely, from $(1-[h])e = 0$, $1-[h] \leq 1 - [\varphi]$ gives the reverse inequality. 
 
Now consider $\varphi(x^*\sigma_t(x)) = \tau(x^*h^{it}xh^{-it} e)$ 
with $x \in M$. 
If the Stieltjes integral expression for $h$ is used as in  
$\displaystyle xh^{-it} e = - \int_0^\infty e^{-ist} d\theta_s(xe)$, we have 
\begin{align*}
-\frac{1}{2\pi} \varphi(x^*\sigma_t(x)) &= \int_0^\infty e^{-ist} d \tau( x^* h^{it} \theta_s(xe) )\\
&= \int_0^\infty e^{-ist} d \tau\Bigl( \theta_s\bigl(x^* \theta_{-s}(h^{it}) xe\bigr) \Bigr)\\ 
&= \int_0^\infty e^{-ist} d(e^{-s} e^{ist}) \tau(x^* h^{it} xe) = (it-1) \tau(^* h^{it} xe) 
\end{align*} 
and then 
\[ 
- \tau( x^* h^{it} xe ) 
= \int_{-\infty}^0 e^{ist} d \tau\Bigl( x^* \theta_s(e) xe \Bigr) 
+ \int_0^\infty e^{ist} d \tau\Bigl( x^* \theta_s(e) xe \Bigr),  
\] 
together with 
\begin{gather*} 
\int_0^\infty e^{ist} d \tau\Bigl( x^* \theta_s(e) xe \Bigr) 
= \int_0^\infty e^{ist} d\Bigl(e^{-s} \tau\bigl( x^* e x\theta_{-s}(e) \bigr) \Bigr)\\
= \int_0^\infty e^{ist} e^{-s} d\tau\Bigl( x^* e x\theta_{-s}(e) \Bigr)
- \int_0^\infty e^{ist} e^{-s} \tau\bigl( x^* e x\theta_{-s}(e) \bigr)\,ds,  
\end{gather*}
reveals that $-\tau(x^*h^{it}xe)$ is analytically extended to a bounded continuous function
\begin{gather*} 
- \tau( x^* h^{iz} xe )
= \int_{-\infty}^\infty e^{isz} d \tau\Bigl( x^* \theta_s(e) xe \Bigr)
= \int_{-\infty}^\infty e^{isz} d\Bigl(e^{-s} \tau( x^*ex\theta_{-s}(e) )\Bigr)\\ 
= \int_{-\infty}^\infty e^{isz} e^{-s} d\tau\Bigl( x^*ex\theta_{-s}(e) \Bigr)
- \int_{-\infty}^\infty e^{isz} e^{-s} \tau\bigl( x^*ex\theta_{-s}(e) \bigr)\,ds. 
\end{gather*}
of $z = t-ir \in \R - i[0,1]$. 

In these and the following calculations, note 
that $\tau( x^*ex\theta_{-s}(e) )$ ($\tau( x^* \theta_s(e) xe )$) 
is positive, increasing (decreasing) and continuous in $s \in \R$, 
whence both $d\tau\Bigl( x^*ex\theta_{-s}(e) \Bigr)$ 
and $-d\tau\Bigl( x^* \theta_s(e) xe \Bigr)$ give rise to positive finite measures on $\R$. 

Consequently, with the notation $\varphi(x^*\sigma_z(x))$ for 
the analytic continuation of $\varphi(x^*\sigma_t(x))$ and,  
with the help of integration-by-parts, we get the expression 
\begin{align*}
&\frac{1}{2\pi} \varphi(x^*\sigma_{t-ir}(x)) 
= (it + r-1)\int_{-\infty}^\infty e^{(it + r -1)s} d\tau\Bigl( x^*ex\theta_{-s}(e) \Bigr)\\ 
&\hspace{4cm}- (it + r-1)\int_{-\infty}^\infty 
e^{(it + r -1)s} \tau\bigl( x^*ex\theta_{-s}(e) \bigr)\,ds\\
&= (it + r) \int_{-\infty}^\infty e^{(it + r -1)s} d\tau\Bigl( x^*ex\theta_{-s}(e) \Bigr) 
- \Bigl[ e^{(it + r -1)s} \tau\bigl( x^*ex\theta_{-s}(e) \bigr) \Bigr]_{-\infty}^\infty. 
\end{align*}
For $0 < r < 1$, we see  
$\displaystyle \lim_{s \to \infty} e^{(it + r -1)s} \tau\bigl( x^*ex\theta_{-s}(e) \bigr) = 0$ and 
\[
\lim_{s \to -\infty} e^{(it + r -1)s} \tau\bigl( x^*ex\theta_{-s}(e) \big) 
= \lim_{s \to -\infty} e^{(it + r)s} \tau\bigl( x^*\theta_s(e)xe \bigr) = 0
\]
at the boundary values and therefore
\[ 
\frac{1}{2\pi} \varphi(x^*\sigma_{t-ir}(x)) = (it + r) 
\int_{-\infty}^\infty e^{(it + r -1)s} d\tau\Bigl( x^*ex\theta_{-s}(e) \Bigr). 
\] 
Since both sides are continuous in $r \in [0,1]$, the equality holds at the 
boundary as well. 
We now compare this expression with 
\begin{align*}
\frac{1}{2\pi} \varphi(\sigma_t(x)x^*) &= \tau( e h^{it}x h^{-it} x^* ) 
= - \int_0^\infty e^{ist} d \tau\Bigl( \theta_s(e) xh^{-it} x^* \Bigr)\\ 
&= - \int_0^\infty e^{ist} d \tau\Bigl( \theta_s\bigl( e x\theta_{-s}(h^{-it}) x^* \bigr) \Bigr)\\ 
&= - \int_0^\infty e^{ist} d(e^{-s-ist}) \tau(e x h^{-it} x^*) 
= (it+1) \tau(e x h^{-it} x^*)\\ 
&= (it+1) \int_{-\infty}^\infty e^{ist} d\tau\Bigl( ex \theta_{-s}(e) x^* \Bigr) 
\end{align*} 
to conclude that $\varphi(x^*\sigma_{t-i}(x)) = \varphi(\sigma_t(x) x^*)$ for $t \in \R$. 

So far we have checked that $h^{it}xh^{-it} = \varphi^{it}x\varphi^{-it}$ for $x \in [\varphi]M[\varphi]$. 
Then $u(t) = h^{it}\varphi^{-it}$ is a unitary in the center of $[\varphi]M[\varphi]$. 
Since each $\varphi^{it}$ commutes with the reduced center, 
$\{ u(t) \}$ is a one-parameter group of unitaries in the reduced algebra. 
Let $u(t) = \int_\R e^{ist}\, E(ds)$ be the spectral decomposition in $[\varphi]M[\varphi]$. 
Then $a_n = \int_{[-n,n]} e^{s/2}\, E(ds)$ is an increasing sequence of positive elements
in the reduced center and $\varphi_n = a_n\varphi a_n \in M_*^+$ satisfies 
$\varphi_n^{it} = h^{it}[a_n] = [a_n]h^{it}$ for $t \in \R$. 
Set $h_n = h[a_n] = [a_n]h$, 
which is also $\tau$-measurable and satisfies $\theta_s(h_n) = e^{-s} h_n$. 
From the equalities 
\[ 
\frac{\varphi_n(x)}{2\pi\mu} = 
\tau( x(1\vee \varphi_n)^{1-\mu} ) 
= \tau( x(1\vee h_n)^{1-\mu} ) 
= \tau( x[a_n]  (1\vee h)^{1-\mu} ) 
= \frac{\varphi(x[a_n])}{2\pi\mu} 
\]
for $x \in M$ and $\mu \geq 1$, one sees that $\varphi_n = \varphi[a_n] = [a_n]\varphi$ 
and then $\varphi_n^{it} = \varphi^{it}[a_n]$ for $t \in \R$. Finally we have 
$h^{it} = \lim_{n \to \infty} h_n^{it} = \lim_{n \to \infty} \varphi^{it} [a_n] = \varphi^{it}$. 

We next check the additivity of the correspondence $h_\varphi \leftrightarrow \varphi$. 
To see this, we first establish the following relation. %averaging relation: 

\begin{Lemma}\label{averaging}
Let $\omega \in M_*^+$ and $\mu>0$. Then  
\[ 
(1\vee \omega)^{-\mu} = \frac{1}{2\pi} \int_\R  \frac{1}{\mu + it} \omega^{it}\, dt 
\]
is in the $\tau$-trace class 
and, for $x \in [\varphi]M$,  we have
\[ 
\tau( h x^*(1\vee \omega)^{-\mu}x ) = \frac{1}{2\pi\mu} \varphi(x^*x).  
% \quad 
% \text{with}
% \quad 
% \int_\R \theta_s(x^*(1\vee \omega)^{-\mu} x)\, ds = \frac{1}{\mu} x^*x. 
\] 
\end{Lemma}
Recall here that
$\displaystyle (1\vee \omega)^{-\mu/2} = \frac{1}{2\pi} \int_\R \frac{1}{it + \mu/2} \omega^{it}\, dt$ 
belongs to $B_+$ in such a way that 
\[ 
(1\vee \omega)^{-\mu/2} \tau^{1/2} 
= \frac{1}{2\pi} \oint_\R \frac{1}{it + (\mu+1)/2} \omega^{it + 1/2}\, dt.  
\] 

The identity is checked as follows: Letting $y = x^* (1\vee \omega)^{-\mu} x$, 
we have 
\begin{align*} 
\tau(hy) &= - \lim_{n \to \infty} 
\int_{-n}^n e^s d \tau\bigl(\theta_s(e)y\bigr)\\ 
&= \lim_{n \to \infty} 
\left( \int_{-n}^n e^s \tau\bigl(\theta_s(e)y\bigr)\, ds - e^n \tau\bigl(\theta_n(e)y\bigr) 
+ e^{-n} \tau\bigl(\theta_{-n}(e)y\bigr) \right)\\
&= \lim_{n \to \infty} \left(
\int_{-n}^nds\, e^s \int_\R dt\, \frac{1}{2\pi(\mu + it)} \tau\bigl(\theta_s(e)x^*\omega^{it} x \bigr)
- \tau\bigl( e\theta_{-n}(y) \bigr) \right)\\
&= \lim_{n \to \infty} \left(
\int_{-n}^nds\, \int_\R dt\, \frac{e^{ist}}{2\pi(\mu +it)} \tau(ex^*\omega^{it} x) 
- \frac{1}{2\pi} \int_\R \frac{e^{int}}{\mu + it} \tau(e x^* \omega^{it} x)\, dt 
\right). 
\end{align*}

By the lemma below, the function $\tau(ex^*\omega^{it} x)/(\mu+it)$ is integrable, whence 
\[ 
\lim_{n \to \infty}  \int_\R \frac{e^{int}}{\mu + it} \tau(e x^* \omega^{it} x)\, dt = 0. 
\]

\begin{Lemma}
We have 
\[ 
\tau(ex^*\omega^{it} x) = \frac{1}{2\pi(1-it)} \varphi(x^*\omega^{it}x\varphi^{-it}). 
\] 
\end{Lemma} 

\begin{proof}
From the expression $\tau( e x^* \omega^{is}x ) = (x e\tau^{1/2}|\omega^{is}x e\tau^{1/2})$ with 
\[ 
\omega^{is}x e\tau^{1/2} = \frac{1}{2\pi} \oint_\R dt\, \frac{1}{i(t-s) + 1/2} 
\omega^{is}x\varphi^{-is} \varphi^{it + 1/2},  
\] 
\begin{align*} 
\tau(e x^*\omega^{is} x) 
&= \frac{1}{(2\pi)^2} 
\int_\R dt\, \frac{1}{-it + 1/2} \frac{1}{i(t-s) + 1/2} 
(x\varphi^{it + 1/2}| \omega^{is} x \varphi^{-is} \varphi^{it + 1/2})\\
&= \frac{1}{(2\pi)^2} 
\int_\R dt\, \frac{1}{-it + 1/2} \frac{1}{i(t-s) + 1/2} 
\varphi(x^*\omega^{is}x\varphi^{-is})\\ 
&= \frac{1}{2\pi} \frac{1}{1-is} \varphi(x^*\omega^{is}x\varphi^{-is}),  
\end{align*}
\end{proof}

To deal with the first term in the last expression of $\tau(hy)$, 
we use the relation $2\pi (\mu + it)^{-1} = g^**g$ for $g(t) = 1/(it + \mu/2)$ to see that 
\[
\int_\R \frac{e^{ist}}{\mu + it} \omega^{it}\, dt
= \int_\R dt'\, e^{-ist'} \overline{g(t')} \omega^{-it'} \int_\R dt\, e^{ist} g(t) \omega^{it} 
\] 
and hence 
\begin{align*}
2\pi \int_\R \frac{e^{ist}}{\mu + it} &(x\xi_0|\omega^{it} x\xi_0)\, dt\\   
&= (\int_\R dt'\, e^{ist'} g(t') \omega^{it'} x\xi_0 | \int_\R dt\, e^{ist} g(t) \omega^{it}x\xi_0)\\ 
&= \sum_j (\int_\R dt'\, e^{ist'} g(t') \omega^{it'} x\xi_0 |\delta_j)
(\delta_j| \int_\R dt\, e^{ist} g(t) \omega^{it}x\xi_0)\\  
&= \sum_j \int_\R dt\, e^{ist} (F_j^**F_j)(t) = \sum_j | \widehat{F_j}(s)|^2,  
\end{align*} 
where $\{ \delta_j\}$ is an orthonormal system in $\sH$ supporting 
vectors $\{ \omega^{it}x\xi_0\}_{t \in \R}$ and 
$F_j(t) = g(t) (\delta_j|\omega^{it} x \xi_0)$ together with their Fourier transforms 
$\widehat{F_j}(s) = \int_\R e^{ist} F_j(t)\, dt$ belong to $L^2(\R)$. 

The Plancherel formula is then applied to each $F_j$ to get 
\begin{align*} 
(2\pi)^2 \tau(hy) 
&= \int_{-\infty}^\infty \sum_j |\widehat{F_j}(s)|^2\, ds 
= \sum_j \int_{-\infty}^\infty |\widehat{F_j}(s)|^2\, ds  
= 2\pi \sum_j \int_\R |F_j(t)|^2\, dt\\ 
&= 2\pi \int_\R \sum_j |F_j(t)|^2\, dt 
= \frac{2\pi}{\mu} (x\xi_0|x\xi_0) = \frac{2\pi}{\mu} \tau(ex^*x). 
\end{align*}

Similarly and more easily, the side identity follows from 
\begin{align*} 
2\pi \int_\R (\xi| \theta_s(y)\xi) 
&= \int_\R ds \int_\R dt\, \frac{1}{\mu + it} (x\xi|\theta_s(\omega^{it})x\xi)\\ 
&= \int_\R ds \int_\R dt\, \frac{e^{-ist}}{\mu + it} (x\xi|\omega^{it}x\xi) 
= \frac{2\pi}{\mu} (\xi|x^*x\xi)
\end{align*}
for each $\xi \in L^2(N)$. 

\begin{Theorem}[Haagerup correspondence]
There is a linear isomorphism between $M_*$ and 
the linear space of $\tau$-measurable operators $h$ on $L^2(N)$ satisfying $\theta_s(h) = e^{-s}h$ and 
so that $\varphi \in M_*^+$ corresponds to 
the analytic generator $h_\varphi$ of the one-parameter group $\{ \varphi^{it}\}$ 
of partial isometries in $N$. 

Moreover the correspondence preserves $N$*-bimodule structures as well as positivity. 
\end{Theorem}

\begin{proof}
The correspondence is already established for positive parts and Lemma~\ref{averaging} is used to 
get the additivity by 
\[ 
\frac{1}{2\pi \mu} \phi(x^*x)  = 
\langle (h_\varphi + h_\psi) x^*(1\vee \omega)^{-\mu} x \rangle 
= \frac{1}{2\pi \mu} (\varphi(x^*x) +\psi(x^*x)).  
\] 
Here $\varphi, \psi \in M_*^+$ and $\phi \in M_*^+$ is specified by $h_\phi = h_\varphi + h_\psi$. 

Once the semilinearity is obtained, the other part is almost automatic. 
The linear extension is well-defined by 
$h_\varphi = h_{\varphi_1} - h_{\varphi_2} + ih_{\varphi_3} - i h_{\varphi_4}$ 
for $\varphi = \varphi_1 - \varphi_2 + i\varphi_3 - i \varphi_4 \in M_*$ with $\varphi_j \in M_*^+$. 
The identity $ah_\varphi a^* = h_{a\varphi a^*}$ for $a \in M$ follows again from Lemma~\ref{averaging} as 
%the averaging relation as 
\[
\tau(aha^* x^*(1\vee \omega)^{-\mu}x) = \tau(ha^* x^*(1\vee \omega)^{-\mu}xa)
= \frac{2\pi}{\mu} \varphi(a^*x^*xa) = \frac{2\pi}{\mu} (a\varphi a^*)(x^*x) 
\]
and then $ah_\varphi b^* = h_{a\varphi b^*}$ by polarization. 
\end{proof}

\end{document}